\documentclass[11pt]{elsarticle}
\usepackage{amssymb,amsfonts,amsmath,mathrsfs,mathtools}
\usepackage{graphicx,epsfig,subfigure}
\usepackage{bm}
\usepackage[margin=1in,papersize={8.5in,11in}]{geometry}
\usepackage{times}
\usepackage{multirow}
\usepackage{color}

\def\vs{\vspace{0.2cm}}

\newtheorem{lemma}{\bf Lemma}[section]

\newtheorem{theorem}{\bf Theorem}[section]

\newtheorem{proposition}{\bf Proposition}[section]

\newenvironment{proof}{{\noindent \bf \em Proof:}}{\hfill$\square$}

\title{Dynamic tensor approximation of high-dimensional nonlinear PDEs}

\begin{document}
\begin{frontmatter}

\author[ucsc]{Alec Dektor}
\author[ucsc]{Daniele Venturi\corref{correspondingAuthor}}
\ead{venturi@ucsc.edu}

\address[ucsc]{Department of Applied Mathematics, University of California Santa Cruz\\ Santa Cruz (CA) 95064}

\cortext[correspondingAuthor]{Corresponding author}

\journal{ArXiv}

\begin{abstract}
We present a new method based on functional tensor 
decomposition and dynamic tensor approximation to
compute the solution of a high-dimensional time-dependent 
nonlinear partial differential equation (PDE).
The idea of dynamic approximation is to project the 
time derivative of the PDE solution onto the 
tangent space of a low-rank functional tensor manifold 
at each time. Such a projection can be computed
by minimizing a convex energy functional over the 
tangent space. This minimization problem yields the
unique optimal velocity vector that allows us to 
integrate the PDE forward in time on a tensor 
manifold of constant rank. 
In the case of initial/boundary value problems defined in 
real separable Hilbert spaces, this procedure yields evolution
equations for the tensor modes in the form of a 
coupled system of one-dimensional time-dependent PDEs.
We apply the dynamic tensor approximation to a 
four-dimensional Fokker--Planck equation with non-constant drift
and diffusion coefficients, and demonstrate its 
accuracy in predicting relaxation to statistical equilibrium.
\end{abstract}
\end{frontmatter}

\section{Introduction}
High-dimensional partial differential equations (PDEs) arise 
in many areas of engineering, physical sciences and 
mathematics.  Classical examples are equations involving 
probability density functions (PDFs) such as 
the Fokker--Planck equation \cite{Risken}, the 
Liouville equation \cite{Venturi_PRS,HeyrimJCP_2014}, or 
the Boltzmann equation \cite{cercignani1988,dimarco2014,TensorBGK}.  
More recently, high-dimensional PDEs 
have also become central to many new areas of 
application such optimal mass transport \cite{Osher2019,Villani}, 
random dynamical systems \cite{Venturi_MZ,Venturi_PRS}, 
mean field games \cite{Weinan2019,Ruthotto2020}, and 
functional-differential equations 
\cite{VenturiSpectral,venturi2018numerical}.

Computing the solution to high-dimensional PDEs is a 
challenging problem that requires approximating 
high-dimensional functions, i.e., the solution to 
the PDE, and then developing appropriate numerical 
schemes to compute such functions accurately.
Classical numerical methods based on tensor product 
representations are not viable in high-dimensions, 
as the number of degrees of freedom grows 
exponentially fast with the dimension. 
To address this problem there have been substantial 
research efforts in recent years on 
approximation theory for high-dimensional systems.
Techniques such as sparse collocation 
\cite{Bungartz,Chkifa,Barthelmann,Foo1,Akil}, 
high-dimensional model representations
\cite{Li1,CaoCG09,Baldeaux},  
deep neural networks \cite{Raissi,Raissi1,Zhu2019} 
and tensor methods 
\cite{khoromskij,Bachmayr,Rodgers_2020,parr_tensor,
Hackbusch_book,Kolda} 
were proposed to mitigate the exponential growth of 
the degrees of freedom, the computational cost and 
memory requirements. In recent work \cite{Dektor_2020}, 
we proposed a new method for solving 
high-dimensional time-dependent PDEs based on 
dynamically orthogonal tensor series expansions. 
The key idea is to represent the solution
in terms of a hierarchy of Schmidt 
decompositions and then enforce dynamic 
orthogonality constraints on the tensor modes. 
In the case of initial/boundary value problems for PDEs 
defined in separable geometries, this procedure yields evolution 
equations for the dynamic tensor modes in the form 
of a coupled system of one-dimensional 
time-dependent PDEs. 

In this paper, we develop an extension of this approach 
based on the functional tensor train (FTT) expansion 
recently proposed by Bigoni, Engsig--Karup and Marzouk
in \cite{Bigoni_2016}. In particular, we prove that FTT, 
combined with the set of hierarchical dynamic orthogonality 
constraints we introduced in \cite{Dektor_2020}, 
defines the {\em best dynamic approximation} 
of the solution to a nonlinear PDE on a smooth tensor 
manifold with constant rank. To describe what we 
mean by best dynamic approximation, 
consider the autonomous PDE
\begin{align}
\frac{\partial u(x,t) }{\partial t} = N(u(x,t)), 
\qquad u(x,0) = u_0(x),
\label{nonlinear-ibvp0} 
\end{align}
where $u:  \Omega \times [0,T] \to\mathbb{R}$ 
is a $d$-dimensional (time-dependent) scalar field defined in 
the domain $\Omega\subseteq \mathbb{R}^d$ 
and $N$ is a nonlinear operator which may 
depend on the spatial variables, and may 
incorporate boundary conditions.
Suppose that at some fixed time $t\in [0,T]$ 
the solution $u(x,t)$ belongs to a smooth 
manifold $\mathcal{M}$ embedded in a 
real Hilbert space $H$.
The best dynamic approximation aims at approximating 
$u(x,t)$ at a later time with a point 
lying on the manifold $\mathcal{M}$ by determining 
the optimal vector in the tangent plane of $\mathcal{M}$
that best approximates $\partial u(x,t)/\partial t$. 
This is achieved by solving the variational problem
\begin{equation}
\label{dyn_approx}
\min_{v(x,t) \in \mathcal{T}_{u(x,t)} \mathcal{M}} \left\| v(x,t) - \frac{\partial u(x,t)}{\partial t} \right\|_H= \min_{v(x,t) \in \mathcal{T}_{u(x,t)} \mathcal{M}} \left\| v(x,t) - N(u(x,t)) \right\|_H. 
\end{equation}
Such an approximation is an infinite-dimensional analogue of 
the dynamical low--rank approximation on Euclidean 
manifolds considered by Lubich {\em et al.} for matrices 
\cite{Lubich_2007,Lubich_2008}, 
Tucker tensors \cite{Lubich_2010}, and 
hierarchical tensors \cite{Lubich_2013,Lubich_2015}.

This paper is organized as follows. 
In section \ref{sec:function_decomp} we briefly 
review the hierarchical Schmidt decomposition 
of multivariate functions (FTT format) and address its 
effective computation. 
In section \ref{sec:manifold_of_tensors} we prove 
that the set of constant-rank FTT tensors is a 
smooth Hilbert manifold, which therefore admits a 
tangent plane at each point. This result generalizes 
\cite[Theorem 4]{h_tucker_geom} to tensor 
manifolds in infinite dimensions.  
In section \ref{sec:optimal_integration_theory} 
we parameterize the tangent space of the Hilbert 
manifold and derive a system of partial differential 
equations for the FTT cores corresponding to a given PDE. 
This system is shown to be the projection of the time derivative 
of the PDE solution onto the tangent space of the 
tensor manifold.
In Section \ref{sec:numerics} we provide a numerical 
demonstration of the dynamic functional tensor train 
approximation for a four-dimensional 
Fokker--Planck equation with non-constant 
drift and diffusion coefficients. 
Finally, the main findings are summarized in 
section \ref{sec:summary}.

\section{Functional tensor train (FTT) decomposition in real separable Hilbert spaces}
\label{sec:function_decomp}
Let $\Omega \subseteq \mathbb{R}^d$ be a 
Cartesian product of $d$ real intervals $\Omega_i=[a_i,b_i]$
\begin{equation}
\begin{aligned}
\Omega &= \bigtimes_{i=1}^d \Omega_i, 
\end{aligned}
\label{Omega}
\end{equation}
$\mu$ a finite product measure on 
$\Omega$
\begin{equation}
\mu(x) = \prod_{i=1}^d \mu_i(x_i),
\end{equation}
and 
\begin{equation}
\label{hilbert_space}
H = L^2_{\mu}(\Omega)
\end{equation}
the standard weighted Hilbert space\footnote{Note that 
the Hilbert space $H$ in equation \eqref{hilbert_space} 
can be equivalently chosen to be a Sobolev space $W^{2,p}$
(see \cite{Dektor_2020} for details).}
of square--integrable functions on $\Omega$. 
In this section we briefly review the functional tensor 
train decomposition  \cite{Bigoni_2016,Gorodetsky2019} 
of a multivariate function $u\in H$
in the setting of hierarchical bi-orthogonal 
series expansions 
\cite{aubry_1,aubry_2,venturi2006,venturi_bi_orthogonal}.
To this end, let $\Omega = \Omega_x \times \Omega_y$, 
$\mu = \mu_x \times \mu_y$  
and $u(x,y) \in L^2_{\mu}(\Omega)$. The operator 
\begin{equation}
\begin{aligned}
T : L^2_{\mu_y}(\Omega_y) &\to L^2_{\mu_x}(\Omega_x) \\
g &\mapsto \int_{\Omega_y} u(x,y) g(y) d\mu_y(y)
\end{aligned}
\label{T}
\end{equation}
is linear, bounded, and compact since $u$ is a 
Hilbert-Schmidt kernel. 
The formal adjoint operator 
of $T$ is given by 
\begin{equation}
\begin{aligned}
T^{\ast} : L^2_{\mu_x}(\Omega_x) &
\to L^2_{\mu_y}(\Omega_y) \\
h &\mapsto \int_{\Omega_x} u(x,y) h(x) d\mu_x(x).
\end{aligned}
\label{T1}
\end{equation}
The composition operator $TT^{\ast}  : L^2_{\mu_x}(\Omega_x) 
\to L^2_{\mu_x}(\Omega_x)$ 
is a self-adjoint compact Hermitian operator. 
The spectrum of $TT^{\ast}$, denoted as 
$\sigma(TT^{\ast})=\{\lambda_1,\lambda_2,\ldots\}$, 
is countable with one accumulation point at $0$, and satisfies 
\begin{equation}
\sum_{i=1}^{\infty} \lambda_i < \infty.
\end{equation} 
The normalized eigenfunction of $TT^{\ast}$ corresponding 
to $\lambda_i$, denoted by $\psi_i(x)$ is an element of 
$L^2_{\mu_x}(\Omega_x)$. The set  
$\{\psi_i\}_{i=1}^{\infty}$ is an 
orthonormal basis of $L^2_{\mu_x}(\Omega_x)$. 
The operator $T^{\ast}T : L^2_{\mu_y}(\Omega_y) 
\to L^2_{\mu_y}(\Omega_y)$ is also 
self-adjoint, compact, and Hermitian, and shares the same 
spectrum as $TT^{\ast}$, i.e., 
$\sigma(TT^{\ast})=  \sigma(T^{\ast}T)$. 
Its eigenfunctions $\{\varphi_i(y)\}_{i=1}^{\infty}$ 
form an orthonormal basis of $L^2_{\mu_y}(\Omega_y)$.
It is a classical result in functional analysis 
that $u(x,y)$ can be expanded as 
(see \cite{Griebel2019,aubry_1,aubry_2})
\begin{equation}
\label{schmidt_decomp}
u(x,y) = \sum_{i=1}^{\infty} \sqrt{\lambda_{i}} \psi_i(x) 
\varphi_i(y).
\end{equation}
The functional tensor train (FTT) decomposition 
recently proposed in \cite{Bigoni_2016} 
can be developed in the setting of hierarchical 
bi-orthogonal expansions as follows.  Let $u\in H$ and set 
$\Omega_x = \Omega_1$ and 
$\Omega_y = \Omega_2 \times \cdots \times \Omega_d$ 
in \eqref{T}-\eqref{T1} to obtain 
\begin{equation}
\label{FTT_level_1}
u(x) = \sum_{\alpha_1=1}^{\infty} \sqrt{\lambda_1(\alpha_1)} 
\psi_1(x_1;\alpha_1) \varphi_1(\alpha_1;x_2,\ldots,x_d).
\end{equation}
Now we let $\Omega_x = \mathbb{N} \times \Omega_2$ and 
$\Omega_y = \Omega_3 \times \cdots \times \Omega_d$ 
and $\tau$ the counting measure on $\mathbb{N}$. 
From the orthonormality 
of $\{\varphi_1(\alpha_1,\cdot)\}_{\alpha_1=1}^{\infty}$ and 
the fact that $u \in L^2_{\mu}(\Omega)$ we have 
\begin{equation}
\begin{aligned}
&\int_{\Omega_x \times \Omega_y} |\sqrt{\lambda_1(\alpha_1)} 
\varphi_1(\alpha_1;x_2,\ldots,x_d)|^2 d\tau(\alpha_1) 
d\mu_2(x_2) \cdots d\mu_d(x_d) \\
&= \sum_{\alpha_1=1}^{\infty} \lambda_1(\alpha_1) 
\int_{\Omega_2 \times \cdots \Omega_d} |
\varphi_1(\alpha_1;x_2,\ldots,x_d)|^2 d\mu_2(x_2) \cdots 
d\mu_d(x_d) \\
&= \sum_{\alpha_1=1}^{\infty} \lambda_1(\alpha_1) < \infty,	
\end{aligned}
\end{equation}
i.e., $(\sqrt{\lambda_1}\varphi_1) \in 
L^2_{\tau \times \mu_2 \times \cdots \times \mu_d}
(X \times Y)$.
Moreover,  $\varphi_1(\alpha_1; x_2,\ldots,x_d)$ can be 
decomposed further by using an expansion of the form \eqref{schmidt_decomp}, i.e.,  
\begin{equation}
\sqrt{\lambda_1(\alpha_1)}\varphi_1(\alpha_1; x_2,\ldots,x_d) = \sum_{\alpha_2=1}^{\infty} 
\sqrt{\lambda_2(\alpha_2)} \psi_2(\alpha_1;x_2;\alpha_2) \varphi_2(\alpha_2; x_3, \ldots, x_d).
\end{equation}
Substituting this expression into \eqref{FTT_level_1} yields
\begin{equation}
u(x) = \sum_{\alpha_1=1}^{\infty}\sum_{\alpha_2=1}^{\infty} 
\sqrt{\lambda_2}(\alpha_2)\psi_1(x_1;\alpha_1) \psi_2(\alpha_1;x_2;\alpha_2) 
\varphi_2(\alpha_2;x_3,\ldots,x_d).
\label{FFT_level_2}
\end{equation}
Proceeding recursively in this manner yields the following FTT expansion 
\begin{equation}
\label{FTT}
u( x) = \sum_{\alpha_1,\ldots,\alpha_{d-1}=1}^{\infty}
\psi_1(\alpha_0;x_1;\alpha_1) \psi_2(\alpha_1;x_2;\alpha_2) \cdots 
\psi_d(\alpha_{d-1};x_d;\alpha_d),
\end{equation}
where $\alpha_0 = \alpha_d = 1$ and $\psi_d(\alpha_{d-1};x_d;\alpha_d) := \sqrt{\lambda_{d-1}(\alpha_{d-1})} \varphi_d(\alpha_{d-1};x_d)$.
By truncating the expansion \eqref{FTT} 
such that the largest singular values 
are retained we obtain
\begin{equation}
\label{FTT_finite}
u_{{TT}}(x) = \sum_{\alpha_0,\ldots,\alpha_{d}=1}^{r}
\psi_1(\alpha_0;x_1;\alpha_1) \psi_2(\alpha_1;x_2;\alpha_2) \cdots 
\psi_d(\alpha_{d-1};x_d;\alpha_d),
\end{equation}
where $r = (1, r_1, \ldots, r_{d-1}, 1)$ is the TT-rank (or rank if 
the TT format is clear from context). 

It is known that the truncated FTT expansion converges optimally with respect to the $L^2_{\mu}(\Omega)$ norm \cite{Bigoni_2016}. More precisely, for any given 
function $u\in L^2_{\mu}(\Omega)$ the FTT approximant \eqref{FTT_finite} minimizes the residual $R_{\text{TT}} = \| u - u_{{TT}} \|_{L^2_{\mu}(\Omega)}$ 
relative to independent variations of the functions 
$\{\psi_i(\alpha_{i-1};x_i;\alpha_{i})\}$ on a tensor 
manifold with constant rank $r$. 
It is convenient to write \eqref{FTT_finite} in a more 
compact form as 
\begin{equation}
u_{{TT}}(x) = \Psi_1(x_1) \Psi_2(x_2) \cdots \Psi_d(x_d),
\end{equation}
where $\Psi_i(x_i)$ is a $r_{i-1} \times r_i$ matrix with entries 
$\left[\Psi_i(x_i)\right]_{jk} = \psi_i(j;x_i;k)$.
The matrix-valued functions $\Psi_i(x_i)$ will be referred 
to as FTT cores. The spatial dependency is clear from 
the subscript of the core so we will often 
suppress the explicit dependence on the spatial variable $x_i$
to simply write $\Psi_i=\Psi_i(x_i)$ and $\psi_i(\alpha_{i-1},\alpha_i)=\psi_i(\alpha_{i-1};x_i;\alpha_i)$.
Rank $r$ FTT decompositions can be computed 
at quadrature points by first discretizing $u$ on a 
tensor product grid and then using a tensor product 
quadrature rule together with known algorithms for 
computing a discrete TT decomposition of a full tensor 
as discussed in \citep{Bigoni_2016}.

At this point we summarize the 
main differences between the FTT series expansion \eqref{FTT_finite} and the series expansions we 
recently developed in \cite{Dektor_2020}. With reference to 
the first level of the hierarchical TT decomposition, i.e., Eq. \eqref{FTT_level_1}, we notice that in the FTT setting 
the functions $\varphi_i(\alpha_i;x_i,\ldots,x_d)$
are not decomposed independently (for each $\alpha_i=1,2,\ldots$) 
as in \cite{Dektor_2020}. Instead, only one bi-orthogonal decomposition is performed on the average 
\begin{equation}
\overline{\varphi}_i(x_{i+1},\ldots,x_d)= \sum_{\alpha_i=1}^{\infty} \varphi_i(\alpha_i; x_{i+1},\ldots,x_d).
\label{averagemode}
\end{equation}
This follows naturally from the assumption 
$\varphi_i(\alpha_i,x_{i+1},\ldots,x_d) \in L^2_{\tau \times \mu_{i+1} \times \cdots \times \mu_d}(\mathbb{N} \times \Omega_{i+1} \times \cdots \times \Omega_d)$, which includes a counting 
measure $\tau$ that yields the summation in \eqref{averagemode}
as part of the inner product. On the other hand, the hierarchical expansion we studied in \cite{Dektor_2020} treats  
$\varphi_i(\alpha_i;x_{i+1},\ldots,x_d)$ as an 
element of $L^2_{\mu_{i+1} \times \cdots \times \mu_d}(\Omega_{i+1} \times \cdots \Omega_d)$ for each $\alpha_i = 1,2,\ldots$ . Hence,  a bi-orthogonal decomposition is performed on $\varphi_i(\alpha_i,x_{i+1},\ldots,x_d)$ 
for each $\alpha_i$. 
Obviously, such decomposition requires many more 
computations but offers more information about 
the spectrum of the multivariate function 
at each level of the TT binary tree. Hereafter we 
proceed by considering the FTT decomposition 
\eqref{FTT_finite}, but note that similar theoretical 
results can also be developed for the hierarchical 
series expansions we studied in \cite{Dektor_2020}.

\section{The manifold of constant rank FTT tensors}
\label{sec:manifold_of_tensors}

In this section we prove that the space of 
constant rank FTT tensors is a smooth manifold, 
which therefore admits a tangent plane at each point. 
The tangent plane will be used in section \ref{sec:optimal_integration_theory} to 
develop an integration theory based on dynamic tensor 
approximation for time-dependent 
nonlinear PDEs. To prove that the space of constant 
rank FTT tensors is a smooth manifold, we follow a 
similar construction as presented in 
\cite{double_do,double_do_stokes}. Closely 
related work was presented in 
\cite{Grassmann_quantum} in relation 
to Slater--type variational spaces in 
many particle Hartree--Fock theory.
Also, the discrete analogues of the infinite-dimensional 
tensor manifolds discussed hereafter were studied 
in detail in \cite{h_tucker_geom,Holtz_2012}. 

Let $\Psi, \tilde{\Psi} \in M_{r_1 \times r_2}
(L_{\mu}^2(\Omega))$, where 
$ M_{r_1 \times r_2}(L_{\mu}^2(\Omega))$ denotes the set of 
$r_1 \times r_2$ matrices with entries in 
$L_{\mu}^2(\Omega)$. Define the matrix 
\begin{equation}
C_{\Psi,\tilde{\Psi}} = \left\langle \Psi^T, \tilde{\Psi} \right\rangle_{L^2_{\mu}(\Omega)} \in M_{r_2 \times r_2}(\mathbb{R}) 
\label{Cdef}
\end{equation}
with entries\footnote{In equations \eqref{Cdef} and 
\eqref{Cdef1} $\langle\cdot,\cdot\rangle_{L^2_\mu(\Omega)}$
denotes the standard inner product in $L^2_\mu(\Omega)$.}
\begin{align}
\left[C_{\Psi,\tilde{\Psi}} \right]_{ij} = 
\sum_{k=1}^{r_1} \left\langle \psi(k;x;i), \tilde{\psi}(k;x;j) \right\rangle_{L^2_{\mu}(\Omega)}.
\label{Cdef1}
\end{align}
Denote by $V_{r_{i-1} \times r_i}^{(i)}$ the set of all 
$\Psi_i \in M_{r_{i-1} \times r_{i}}(L^2_{\mu_i}(\Omega_i))$ with the property that $C_{\Psi_i,\Psi_i}$ is 
invertible.
We are interested in the following subset of $L^2_{\mu}(\Omega)$ consisting of rank-$r$ FTT 
tensors in $d$ dimensions  
\begin{equation}
\begin{aligned}
\mathfrak{T}_{r}^{(d)} &= \{ u \in L^2_{\mu}(\Omega):\quad u 
= \Psi_1 \Psi_2 \cdots \Psi_d , \quad 
\Psi_i \in V_{r_{i-1} \times r_i}^{(i)}, \quad \forall i = 1,2,\ldots,d
\}.
\end{aligned}
\end{equation}
The set 
\begin{equation}
V = V_{r_0 \times r_1}^{(1)} \times 
V_{r_1 \times r_2}^{(2)} \times \cdots 
\times V_{r_{d-1} \times r_d}^{(d)}
\label{V}
\end{equation} 
can be interpreted as a latent space 
for $\mathfrak{T}_{r}^{(d)}$ via 
the mapping 
\begin{equation}
\label{TT_submersion}
\begin{aligned}
\pi : V &\to \mathfrak{T}_{r}^{(d)} \qquad \pi(\Psi_1,\Psi_2,\ldots,\Psi_d) = \Psi_1 \Psi_2 \cdots \Psi_d.
\end{aligned}
\end{equation}
Any tensor $u \in \mathfrak{T}_{r}^{(d)}$ has many representations in $V$, that is the map $\pi(\cdot)$ 
is not injective. The purpose of the following Lemma \ref{lemma:invertible_cov_mats_gen}
and Proposition \ref{lemma:two_dd_decomps} is 
to characterize all elements of the space $V$ 
which have the same image under $\pi$. 

\begin{lemma}
\label{lemma:invertible_cov_mats_gen}
If $\{\psi(\alpha_i;x;\alpha_j)\}_{\alpha_j=1}^{r},\{\tilde{\psi}(\alpha_i;x;\alpha_j)\}_{\alpha_j=1}^r$ are two bases for the same finite 
dimensional subspace of $L^2_{\tau \times \mu}(\mathbb{N} \times \Omega)$ 
then the matrix $C_{\Psi,\tilde{\Psi}}$ defined in \eqref{Cdef} 
is invertible.
\end{lemma}
\begin{proof}
The matrix under consideration is given by 
\begin{equation}
C_{\Psi,\tilde{\Psi}} = 
\begin{bmatrix}
\displaystyle\sum_{k=1}^r \left\langle \psi(k;x;1), \tilde{\psi}(k;x;1) \right\rangle_{L^2_{\mu}(\Omega)} & \cdots & 
\displaystyle\sum_{k=1}^r \left\langle \psi(k;x;1), \tilde{\psi}(k;x;r) \right\rangle_{L^2_{\mu}(\Omega)} \\
\vdots & \ddots & \vdots \\
\displaystyle\sum_{k=1}^r \left\langle \psi(k;x;r), \tilde{\psi}(k;x;1) \right\rangle_{L^2_{\mu}(\Omega)} & \cdots & 
\displaystyle\sum_{k=1}^r \left\langle \psi(k;x;r), \tilde{\psi}(k;x;r) \right\rangle_{L^2_{\mu}(\Omega)}
\end{bmatrix}.
\end{equation}
We will show that the columns of this matrix are linearly 
independent. To this end, consider the linear equation  
\begin{equation}
\sum_{i=1}^r v_i \begin{bmatrix}
\displaystyle\sum_{k=1}^r \left\langle \psi(k;x;1), \tilde{\psi}(k;x;i) \right\rangle_{L^2_{\mu}(\Omega)} \\
\vdots \\
\displaystyle\sum_{k=1}^r \left\langle \psi(k;x;r), \tilde{\psi}(k;x;i) \right\rangle_{L^2_{\mu}(\Omega)}
\end{bmatrix} = 0 \qquad v_i \in \mathbb{R}
\end{equation} 
the $p$-th row of which reads 
\begin{equation}
\label{orthogonality_of_two_bases}
\displaystyle\sum_{k=1}^r \left\langle \psi(k;x;p), \sum_{i=1}^r v_i \tilde{\psi}(k;x;i) \right\rangle_{L_{\mu}^2(\Omega)} = 0  \qquad p = 1, \ldots, r.
\end{equation}
If not all the $v_i$ are equal to zero then \eqref{orthogonality_of_two_bases} implies that 
$\displaystyle\sum_{i=1}^r v_i \tilde{\psi}(k;x;i)$ is 
orthogonal to $\psi(k;x;p)$ in 
$L^2_{\tau \times \mu}(\mathbb{N} \times \Omega)$ 
and therefore linearly independent 
for all $p = 1,\ldots, r$. This contradicts 
the assumption that $\{\psi(i,j)\}_{j=1}^{r},\{\tilde{\psi}(i,j)\}_{j=1}^r$ 
span the same finite dimensional subspace of 
$L^2_{\tau \times \mu}(\mathbb{N} \times \Omega)$. Hence $v_i$ are zero for every $i=1,\ldots,r$.

\end{proof}
\begin{proposition}
\label{lemma:two_dd_decomps}
Let $\{\Psi_i\}_{i=1}^d$, $\{\tilde{\Psi}_i\}_{i=1}^d$ be elements of $V$. Then 
\begin{equation}
\pi(\Psi_1,\ldots,\Psi_d)=\pi(\tilde{\Psi}_1,\ldots,\tilde{\Psi}_d) 
\label{pi}
\end{equation}
if and only if there exist matrices $P_i \in \text{GL}_{r_i \times r_i}(\mathbb{R})$ ($i = 0,1,\ldots,d$) 
such that $\Psi_i = P_{i-1}^{-1} \tilde{\Psi}_i P_i$ 
with $P_0, P_d = 1$.
\end{proposition}
\begin{proof}
To prove the forward implication we proceed 
by induction on $d$. For $d = 2$ we have that  
\begin{equation}
\label{2d_decomp}
\Psi_1 \Psi_2 = \tilde{\Psi}_1 \tilde{\Psi}_2 
\end{equation}
implies 
\begin{equation}
\Psi_1 = \tilde{\Psi}_1 C_{\tilde{\Psi}_2^T,\Psi_2^T} C_{\Psi_2^T,\Psi_2^T}^{-1}.
\end{equation}
Set $P_1 = C_{\tilde{\Psi}_2^T,\Psi_2^T} C_{\Psi_2^T,\Psi_2^T}^{-1}$
which is invertible since it is a change of basis matrix. 
Substituting $\Psi_1 = \tilde{\Psi}_1 P_1$ into 
\eqref{2d_decomp} we see that 
\begin{equation}
\tilde{\Psi}_1 P_1 \Psi_2 = \tilde{\Psi}_1 \tilde{\Psi}_2 
\end{equation}
which implies 
\begin{equation}
\Psi_2 = P_1^{-1} \tilde{\Psi}_2.
\end{equation}
This proves the proposition for $d =2$. 
Suppose that the proposition holds true for 
$d-1$ and that 
\begin{equation}
\label{dd_decomp}
\Psi_1 \cdots \Psi_d 
= \tilde{\Psi}_1 \cdots \tilde{\Psi}_d.
\end{equation}
Then,
\begin{equation}
\Psi_1 \cdots \Psi_{d-1} = \tilde{\Psi}_1 \cdots \tilde{\Psi}_{d-1}
C_{\tilde{\Psi}_d^T,\Psi_d^T} C_{\Psi_d^T,\Psi_d^T}^{-1},
\end{equation}
and we are gauranteed the existence of invertible 
matrices $P_1, \ldots, P_{d-2}$ such that 
\begin{equation}
\begin{aligned}
\Psi_1 &= \tilde{\Psi}_1 P_1, \\
& \hspace{0.2cm}\vdots \\
\Psi_{d-2} &= P_{d-3}^{-1} \tilde{\Psi}_{d-2} P_{d-2}, \\
\Psi_{d-1} &= P_{d-2}^{-1} \tilde{\Psi}_{d-1} C_{\tilde{\Psi}_d^T,\Psi_d^T} C_{\Psi_d^T,\Psi_d^T}^{-1}.
\end{aligned}
\label{Eq}
\end{equation}
Let $P_{d-1} =  C_{\tilde{\Psi}_d^T,\Psi_d^T} C_{\Psi_d^T,\Psi_d^T}^{-1}$. Substituting equation \eqref{Eq} 
into \eqref{dd_decomp} yields 
\begin{equation}
\begin{aligned}
\tilde{\Psi}_1 \cdots \tilde{\Psi}_{d-2} \tilde{\Psi}_{d-1} P_{d-1} \Psi_d &= 
\tilde{\Psi}_1 \cdots \tilde{\Psi}_d, \\
P_{d-1} \Psi_d &= \tilde{\Psi}_d,
\end{aligned}
\end{equation}
from which it follows that $P_{d-1}$ is invertible and 
\begin{equation}
\Psi_{d} = P_{d-1}^{-1} \tilde{\Psi}_d.
\end{equation}
This completes the proof.

\end{proof}

\vs
\noindent
With Proposition \ref{lemma:two_dd_decomps} in 
mind we define the group\footnote{In equation \eqref{G}
$\text{GL}_{r_1\times r_i}(\mathbb{R})$ 
denotes the general linear group of $r_i\times r_i$ 
invertible matrices with real entries, together with the operation of 
ordinary matrix multiplication.}
\begin{equation}
\label{G}
G = \text{GL}_{r_1 \times r_1}(\mathbb{R}) \times \text{GL}_{r_2 \times r_2}(\mathbb{R}) \times \cdots \times \text{GL}_{r_{d-1} \times r_{d-1}}(\mathbb{R}) 
\end{equation} 
with group operation given by component-wise 
matrix multiplication. Let $G$ act on $V$ by 
\begin{equation}
(P_1,\ldots,P_{d-1}) \cdot (\Psi_1,\ldots,\Psi_d) = (\Psi_1 P_1, P_1^{-1} \Psi_2 P_2, \ldots, P_{d-1}^{-1} \Psi_d ) 
\end{equation}
for all $(P_1,\ldots,P_{d-1}) \in G$ and $(\Psi_1,\ldots,\Psi_d) \in V$. 
It is easy to see that this is action is free and transitive 
making $G$, $V$, $\mathfrak{T}_r^{(d)}$ and 
$\pi$ a principal $G$-bundle \cite{Rudolph2017}.
In particular $V/G$ is 
isomorphic to $\mathfrak{T}_{r}^{(d)}$ which 
allows us to equip $\mathfrak{T}_{r}^{(d)}$ 
with a manifold structure. Thus, we can define its 
tangent space $\mathcal{T}_u\mathfrak{T}_{r}^{(d)}$ 
at a point $u \in \mathfrak{T}_{r}^{(d)}$. We 
characterize such tangent space as the equivalence 
classes of velocities of smooth curves passing 
through the point $u$   
\begin{equation}
\label{FTT_dd_tangent_space}
\mathcal{T}_u\mathfrak{T}_{r}^{(d)} = \left\{\gamma'(s)\vert_{s=0}: \quad \gamma \in \mathcal{C}^1\left( (-\delta,\delta) , \mathfrak{T}_{r}^{(d)} \right), \quad  
\gamma(0) = u \right\}.
\end{equation}
Here $\mathcal{C}^1\left( (-\delta,\delta) , \mathfrak{T}_{r}^{(d)} \right)$ is the space of continuously 
differentiable functions from the interval 
$(-\delta, \delta)$ to 
the space of constant rank FTT tensors 
$\mathfrak{T}_{r}^{(d)}$.
We conclude this section with the following Lemma which 
singles out a particular representation 
of $u \in \mathfrak{T}_{r}^{(d)}$ for which 
the matrices $C_{\Psi_i,\Psi_i}$ 
(see Eqs. \eqref{Cdef}-\eqref{Cdef1}) are diagonal.
\begin{lemma}
\label{lemma:orthogonal_FTT_cores}
Given any FTT tensor $u \in \mathfrak{T}_{r}^{(d)}$ 
there exist $\Psi_i \in V_{r_{i-1} \times r_i}^{(i)}$ $(i = 1, 2, \ldots, d)$ such that $u=\Psi_1 \Psi_2 \cdots \Psi_d$ and $C_{\Psi_i,\Psi_i} = I_{r_i \times r_i }$ for 
all  $i = 1,\ldots , d-1$. 
\end{lemma}
\begin{proof}
Let us first represent $u \in \mathfrak{T}_{r}^{(d)}$ relative 
to the tensor cores $\{\tilde{\Psi}_1,\ldots,\tilde{\Psi}_d\}$. 
Since $C_{\tilde{\Psi}_1,\tilde{\Psi}_1}$ is symmetric there 
exists an orthogonal matrix $P_1$ such that 
$P_1^T C_{\tilde{\Psi}_1,\tilde{\Psi}_1}P_1 = \Lambda_1$ is diagonal. 
Set ${\Psi}_1 = \tilde{\Psi}_1 P_1 \Lambda_1^{-1/2}$ and 
$\hat{\Psi}_2 = \Lambda_1^{1/2} P_1^T \tilde{\Psi}_2$ so that 
$C_{{\Psi}_1,{\Psi}_1} = I_{r_1 \times r_1}$ and 
${\Psi}_1 \hat{\Psi}_2 \tilde{\Psi}_3 \cdots \tilde{\Psi}_d = 
\tilde{\Psi}_1 \cdots \tilde{\Psi}_d$.
The matrix $C_{\hat{\Psi}_2,\hat{\Psi}_2}$ is symmetric so there 
exists an orthogonal matrix $P_2$ such that 
$P_2^T C_{\hat{\Psi}_2,\hat{\Psi}_2} P_2 = \Lambda_2$ is diagonal. 
Set ${\Psi}_2 = \hat{\Psi}_2 P_2 \Lambda_2^{-1/2}$ and 
$\hat{\Psi}_3 = \Lambda_2^{1/2}P_2^T \tilde{\Psi}_3$ so that 
$C_{{\Psi}_2,{\Psi}_2} = I_{ r_2 \times r_2}$ and 
${\Psi}_1 {\Psi}_2 \hat{\Psi}_3 \tilde{\Psi}_4 \cdots \tilde{\Psi}_d = 
\tilde{\Psi}_1 \cdots \tilde{\Psi}_d$.
Proceed recursively in this way until 
${\Psi}_1 {\Psi}_2 \cdots {\Psi}_{d-1} \hat{\Psi}_d = 
\tilde{\Psi}_1 \cdots \tilde{\Psi}_d$ with 
$C_{{\Psi}_i, {\Psi}_i} = I_{ r_i \times r_i}$, $i = 1,\ldots,d-1$. 
It is easy to check that the collection of 
cores 
$\{\Psi_1,\ldots,\Psi_d\}$ satisfies the 
conclusion of the Lemma.

\end{proof}

\section{Dynamical approximation of PDEs on FTT tensor manifolds with constant rank}
\label{sec:optimal_integration_theory}

Computing the solution to high-dimensional PDEs 
has become central to many new areas of 
application such as optimal mass transport \cite{Osher2019,Villani}, 
random dynamical systems \cite{Venturi_MZ,Venturi_PRS}, 
mean field games \cite{Weinan2019,Ruthotto2020}, and 
functional-differential equations \cite{VenturiSpectral,venturi2018numerical}.
In an abstract setting, such PDEs involve the 
computation of a function $u(x,t)$ governed 
by an autonomous evolution equation
\begin{align}
\begin{cases}
\displaystyle\frac{\partial u }{\partial t} = N(u), \vspace{0.1cm} \\
u(x,0) = u_0(x),
\end{cases}
\label{nonlinear-ibvp} 
\end{align}
where $u:  \Omega \times [0,T] \to\mathbb{R}$ 
is a $d$-dimensional (time-dependent) scalar field defined in 
the domain $\Omega\subseteq \mathbb{R}^d$ (see Eq. 
\eqref{Omega}) and $N$ is a nonlinear operator which 
may depend on the spatial variables and may 
incorporate boundary conditions.

We are interested in computing the 
{\em best dynamic approximation} 
of the solution to \eqref{nonlinear-ibvp} on the tensor 
manifold $\mathfrak{T}_{r}^{(d)}$ for all $t\geq 0$.
Such an approximation aims at determining the 
vector in the tangent plane of $\mathfrak{T}_{r}^{(d)}$ 
at the point $u$ that best approximates 
$\partial u/\partial t$ for each 
$u\in \mathfrak{T}_{r}^{(d)}$. 
One way to obtain the optimal vector in the tangent plane 
is by orthogonal projection which we now describe. 
For each $u \in L^2_{\mu}(\Omega)$ the tangent space 
$\mathcal{T}_u L^2_{\mu}(\Omega)$ is canonically isomorphic 
to $L^2_{\mu}(\Omega)$. Moreover, for 
each $u \in \mathfrak{T}_r^{(d)}$ the normal space 
to $\mathfrak{T}_r^{(d)}$ at the point $u$, denoted by 
$\mathcal{N}_{u} \mathfrak{T}_r^{(d)}$, 
consists of all vectors in $L^2_{\mu}(\Omega)$ that are 
orthogonal to $\mathcal{T}_{u} \mathfrak{T}_r^{(d)}$ 
with respect to the inner product in $L^2_{\mu}(\Omega)$.
The space $\mathcal{T}_u \mathfrak{T}_r^{(d)} \subseteq L^2_{\mu}(\Omega)$ is finite-dimensional
and therefore it is closed. Thus, 
for each $u \in \mathfrak{T}_r^{(d)}$ 
the space $L^2_{\mu}(\Omega)$ admits the decomposition 
\begin{equation}
L^2_{\mu}(\Omega) = \mathcal{T}_{u} \mathfrak{T}_r^{(d)} \oplus \mathcal{N}_{u} \mathfrak{T}_r^{(d)}.
\end{equation}
Assuming that the solution $u(x,t)$ to the 
PDE \eqref{nonlinear-ibvp} lives on 
the manifold $\mathfrak{T}_r^{(d)}$ at time $t$, we have that 
its velocity $\partial u/\partial t = N(u)$ can be 
decomposed uniquely into a tangent component 
and a normal component with respect 
to $\mathfrak{T}^{(d)}_r$, i.e.,  
\begin{equation}
\label{tangent_and_normal}
N(u) = v + w, \qquad v \in \mathcal{T}_u \mathfrak{T}_r^{(d)} , \quad w \in \mathcal{N}_u \mathfrak{T}_r^{(d)}.
\end{equation}
The orthogonal projection we are interested in computing for the 
best dynamic approximation is 
\begin{equation}
\label{orthogonal_projection}
\begin{aligned}
P_u : L^2_{\mu}(\Omega) &\to \mathcal{T}_u \mathfrak{T}_r^{(d)}, \\
N(u) &\mapsto P_u  N(u).
\end{aligned}
\end{equation}
In practice, we will compute the image of such a projection 
by solving the following minimization problem 
over the tangent space of 
$\mathfrak{T}^{(d)}_r$ at $u$ 
\begin{equation}
\label{min_problem}
\min_{v(x,t) \in  \mathcal{T}_{u(x,t)} \mathfrak{T}_{r}^{(d)}} \left\| v(x,t) - \frac{\partial u(x,t)}{\partial t} \right\|^2_{L^2_{\mu}(\Omega)} = \min_{v(x,t) \in  \mathcal{T}_{u(x,t)} \mathfrak{T}_{r}^{(d)}} \left\| v(x,t) - N(u(x,t)) \right\|^2_{L^2_{\mu}(\Omega)}
\end{equation}
for each fixed $t\in [0,T]$.
From an optimization viewpoint 
the following proposition establishes the 
existence and uniqueness of the optimal 
tangent vector.
\begin{proposition}
If $N(u) \not\in \mathfrak{T}_{r}^{(d)}$ then there 
exists a unique solution to the minimization problem 
\eqref{min_problem}, i.e., a unique global minimum.
\label{prop:globalminimum}
\end{proposition}
\begin{proof}
We first notice that the feasible set $\mathcal{T}_{u} 
\mathfrak{T}_{r}^{(d)}$ is a real vector space 
and thus a convex set.
Next we show that the functional 
$F[v] = \| v - N(u) \|^2_{L^2_{\mu}(\Omega)}$ is 
strictly convex. Indeed, take $v_1, v_2 \in \mathcal{T}_u \mathfrak{T}_{r}^{(d)}$ distinct and $q \in (0,1)$.  Then 
\begin{equation}
\label{strictly_convex_functional}
\begin{aligned}
\left(F[q v_1 + (1-q)v_2]\right)^{\frac{1}{2}} &= \|q v_1 + (1-q) v_2 - N(u) \|_{L^2_{\mu}(\Omega)} \\
&= \|q (v_1 - N(u) ) + (1-q) \left( (v_2 - N(u) \right) \|_{L^2_{\mu}(\Omega)} \\
&\leq q \| v_1 - N(u) \| + (1-q) \| v_2 - N(u) \|_{L^2_{\mu}(\Omega)},
\end{aligned}
\end{equation}
with equality if and only if there exists an $\alpha > 0$ such 
that $q (v_1 - N(u) ) = \alpha (1-q) ( v_2 - N(u) )$. However, 
this implies that $v_1 - \beta v_2 = (1-\beta) N(u)$ 
for some real number $\beta$, whence 
$N(u) \in \mathcal{T}_{u} \mathfrak{T}_{r}^{(d)}$. 
Therefore if $N(u) \not\in \mathcal{T}_{u} 
\mathfrak{T}_{r}^{(d)}$ then the inequality in 
\eqref{strictly_convex_functional} is strict and 
the functional $\left(F[v]\right)^{\frac{1}{2}}$ is 
strictly convex. Since the function $x^2$ is strictly 
increasing on the image of $F^{\frac{1}{2}}$ 
it follows that $F$ is strictly convex and thus 
admits a unique global minimum over the 
feasible set $\mathcal{T}_{u} \mathfrak{T}_{r}^{(d)}$. 

\end{proof}

\vs
\noindent
It can easily be shown that the unique solution to 
the optimization problem \eqref{min_problem} is 
$P_u N(u)$.
Next, we will use this optimization framework 
for computing the best tangent vector 
to integrate the PDE \eqref{nonlinear-ibvp} forward 
in time on the manifold $\mathfrak{T}_r^{(d)}$.
To this end, let us first assume 
that the initial condition $u_0\in \mathfrak{T}_{r}^{(d)}$. 
If not, $u_0$ can be projected onto 
$\mathfrak{T}_{r}^{(d)}$ using the methods 
described in section \ref{sec:function_decomp}. 
In both cases, this allows us to represent $u_0(x)$ as 
\begin{equation}
u_0(x) = \Psi_1(0) \Psi_2(0) \cdots \Psi_d(0),
\end{equation}
with $C_{\Psi_i(0),\Psi_i(0)} = I_{r_i \times r_i}$ 
for $i = 1,2,\ldots,d-1$. A representation of this form 
(with diagonal matrices $C_{\Psi_i(0),\Psi_i(0)}$) 
always exists thanks to Lemma 
\ref{lemma:orthogonal_FTT_cores}. 
To compute the unique solution of \eqref{min_problem}, 
we expand an arbitrary curve $\gamma(s)$ of class 
$\mathcal{C}^1$ on the manifold $\mathfrak{T}_{r}^{(d)}$ 
passing through the point $u\in \mathfrak{T}_{r}^{(d)}$ at 
$s=0$ in terms of $s$-dependent FTT cores. This yields 
\begin{equation}
\begin{aligned}
\gamma(s) &= \Psi_1(s) \cdots \Psi_d(s) \\
&= \sum_{\alpha_0,\ldots,\alpha_d=1}^{\bm r} \psi_1(s;\alpha_0,\alpha_1) \psi_2(s;\alpha_1,\alpha_2) \cdots \psi_d(s;\alpha_{d-1},\alpha_d),
\end{aligned}
\end{equation}
which allows us to represent any element of the tangent 
space $\mathcal{T}_{u} \mathfrak{T}_{r}^{(d)}$  at $u$ 
as
\begin{equation}
\label{velocity_in_terms_of_cores}
v = \frac{\partial }{\partial s} \left[ \Psi_1(s) \cdots \Psi_d(s) \right]_{s=0},
\end{equation} 
with $\Psi_1(0) \cdots \Psi_d(0) = u$.
At this point we notice that minimizing the 
functional in \eqref{min_problem} 
over the tangent space $\mathcal{T}_{u} \mathfrak{T}_{r}^{(d)}$ 
is equivalent to minimizing the same functional over the velocity 
of each of the FTT cores. For notational 
convenience, hereafter we omit evaluation at $s=0$ 
of all quantities depending on 
the curve parameter $s$. For example, we will write
\begin{equation}
\psi_i(\alpha_{i-1},\alpha_i)=\left.\psi_i(s;\alpha_{i-1},\alpha_i)
\right|_{s=0},\qquad 
\frac{\partial \psi_i(\alpha_{i-1},\alpha_i)}{\partial s}=
\left.\frac{\partial \psi_i(s;\alpha_{i-1},\alpha_i)}{\partial s}
\right|_{s=0}.
\end{equation} 
With this notation, the minimization 
problem \eqref{min_problem} is equivalent to 
\begin{equation}
\label{min_over_cores}
\min_{\frac{\partial \psi_1(\alpha_0,\alpha_1)}{\partial s}, \ldots, 
\frac{\partial{\psi}_d(\alpha_{d-1},\alpha_d)}{\partial s}} \left\|  
\frac{\partial }{\partial s} \left( \displaystyle\sum_{\alpha_0,\ldots,
\alpha_d=1}^{r} \psi_1(\alpha_0,\alpha_1) \psi_2(\alpha_1,
\alpha_2) \cdots \psi_d(\alpha_{d-1},\alpha_d) \right)- N(u) 
\right\|^2_{L^2_{\mu}(\Omega)}.
\end{equation}
Of course, in view of Lemma \ref{lemma:two_dd_decomps} 
one curve $\gamma(s)$ has many different expansions 
in terms of $s$-dependent FTT cores, which 
can be mapped into one another via collections of 
$s$-dependent invertible matrices. 
From Lemma \ref{lemma:orthogonal_FTT_cores} it is clear that 
any curve $\gamma(s) \in \mathcal{C}^1
\left((-\delta,\delta), \mathfrak{T}_{r}^{(d)}\right)$ 
passing through $u$ at $s=0$ admits the FTT decomposition 
$\gamma(s) = \Psi_1(s) \cdots \Psi_d(s)$, where all auto-correlation matrices $C_{\Psi_i,\Psi_i}(s)$ ($i = 1,\ldots,d-1$) 
are identity matrices for all for all $s \in (-\delta,\delta)$, i.e., 
\begin{equation}
\label{general_constraint}
C_{\Psi_i,\Psi_i}(s) = \left\langle \Psi_i^T(s), \Psi_i(s) \right\rangle_{L^2_{\mu_i}(\Omega_i)}=I_{r_i \times r_i}.
\end{equation}
Differentiating \eqref{general_constraint} with respect 
to $s$ yields  
\begin{equation}
\label{skew_symm_DO_constraint}
\left\langle \frac{\partial \Psi_i^T(s)}{\partial s}, \Psi_i(s) \right\rangle_{L^2_{\mu_i}(\Omega_i)} = - \left\langle \Psi_i^T(s), \frac{\partial \Psi_i(s)}{\partial s} \right\rangle_{L^2_{\mu_i}(\Omega_i)},
\end{equation}
which is attained when 
\begin{equation}
\label{DO_constraints_multilevel}
\left\langle \frac{\Psi_i^T(s)}{\partial s}, \Psi_i(s) \right\rangle_{L^2_{\mu_i}(\Omega_i)} = 0_{r_i \times r_i}, \qquad \forall s \in (-\delta,\delta).
\end{equation} 
Enforcing \eqref{DO_constraints_multilevel} and 
prescribing \eqref{general_constraint} at $s=0$ 
is equivalent to enforcing \eqref{general_constraint} for all 
$s \in (-\delta,\delta)$.
With this characterization of continuously differentiable 
curves $\gamma(s)$ passing through $u\in \mathfrak{T}^{(d)}_r$, we can recast the minimization problem \eqref{min_over_cores} 
in terms of FTT cores constrained\footnote{We are also assuming $u = \Psi_1 \cdots \Psi_d$ is such that $C_{\Psi_i,\Psi_i} = I_{r_i \times r_i}$ for $i = 1,\ldots,d-1$ so the constraint \eqref{general_constraint} is satisfied at $s=0$.} by \eqref{DO_constraints_multilevel} 
\begin{equation}
\label{constrained_min_problem}
\begin{cases}
\displaystyle\min_{\frac{\partial \psi_1(\alpha_0,\alpha_1)}{\partial s}, \ldots, 
\frac{\partial{\psi}_d(\alpha_{d-1},\alpha_d)}{\partial s}}
 \left\|  \frac{\partial }{\partial s} \left[ \displaystyle\sum_{\alpha_0,\ldots,\alpha_d=1}^{ r} \psi_1(\alpha_0,\alpha_1) \psi_2(\alpha_1,\alpha_2) \cdots \psi_d(\alpha_{d-1},\alpha_d) \right] - N(u) \right\|^2_{L^2_{\mu}(\Omega)}\vspace{0.1cm}\\ 
\text{subject to:}\vspace{0.2cm}\\
\quad\displaystyle\left\langle \frac{\partial \psi_i(\alpha_{i-1},\alpha_i)}{\partial s}, \psi_i(\alpha_{i-1},\beta_i) \right\rangle_{L^2_{\tau \times \mu_i}(\mathbb{N} \times \Omega_i)} = 0 \quad \forall i = 1,2,\ldots, d-1,\quad  
\alpha_i,\beta_i = 1,2,\ldots,r_i,
\end{cases}
\end{equation}
which by the discussion above still has the entire 
tangent space $\mathcal{T}_{u} \mathfrak{T}_{r}^{(d)}$ 
as the feasible set. The minimization problem
\eqref{constrained_min_problem} is a convex 
optimization problem subject to linear equality 
constraints, which therefore is still convex. Hence, any 
local minimum is also a global minimum. Moreover 
a minimum of  \eqref{constrained_min_problem} 
provides the velocities of FTT cores which allow 
for the construction of the unique global minimum 
to the optimization problem \eqref{min_problem} 
via equation \eqref{velocity_in_terms_of_cores}. 
To solve \eqref{constrained_min_problem} it is 
convenient to construct an action functional 
$\mathcal{A}$ that introduces the constraints 
via Lagrange multipliers $\lambda_{\alpha_i \beta_i}^{(i)}$
\begin{equation}
\label{constrained_min_problem_lagrangian}
\begin{aligned}
\mathcal{A}&\left(\frac{\partial 
\psi_1(\alpha_0,\alpha_1)}{\partial s}, \ldots, \frac{\partial{\psi}_d(\alpha_{d-1},\alpha_d)}{\partial s}\right) 
= \\ 
&\left\|  \frac{\partial }{\partial s} \left[ \displaystyle\sum_{\alpha_0,\ldots,\alpha_d=1}^{ r} \psi_1(\alpha_0,\alpha_1) \psi_2(\alpha_1,\alpha_2) \cdots \psi_d(\alpha_{d-1},\alpha_d) \right] - N(u) \right\|^2_{L^2_{\mu}(\Omega)} + \\
&\sum_{i=1}^{d-1} \sum_{\alpha_i,\beta_i=1}^{ r_i} 
\lambda_{\alpha_i \beta_i}^{(i)} \left\langle \frac{\partial \psi_i(\alpha_{i-1},\alpha_i)}{\partial s},
\psi_i(\alpha_{i-1},\beta_i)\right\rangle_{L^2_{\tau \times \mu_i}(\mathbb{N}\times \Omega_i)}.
\end{aligned}
\end{equation}
At this point, we have all elements to formulate the FTT 
propagator for the nonlinear PDE \eqref{nonlinear-ibvp}, 
which is the system of Euler-Lagrange equations 
corresponding to the unique global minimum of \eqref{constrained_min_problem_lagrangian}.
Such propagator allows us to determine the best dynamic
approximation of the solution to \eqref{nonlinear-ibvp}
on a FTT tensor manifold with constant rank.

\begin{theorem}
\label{thm:evolution_equations}
The unique global minimum of the functional \eqref{constrained_min_problem_lagrangian} is 
attained at FTT tensor cores satisfying the PDE system 
\begin{equation}
\label{DO-FTT_system}
\begin{aligned}
\frac{\partial \Psi_1}{\partial t} &= \left[ \left\langle N(u),\Phi_1^T \right\rangle_{2,\ldots,d} 
- \Psi_1 \left\langle \left\langle \Psi_1^T, N(u) \right\rangle_{1}, \Phi_1^T \right\rangle_{2,\ldots,d} \right] C_{\Phi_1^T,\Phi_1^T}^{-1}, \\
\frac{\partial \Psi_k}{\partial t} &= \left[ \left\langle \left\langle \Psi_{k-1}^T \cdots \Psi_1^T, N(u) \right\rangle_{1,\ldots,k-1} , \Phi_k^T \right\rangle_{k+1,\ldots,d} - \right.\\
&\quad\quad \left.\Psi_k \left\langle \left\langle \Psi_k^T \cdots \Psi_1^T, N(u) \right\rangle_{1,\ldots,k} \Phi_k^T \right\rangle_{k+1,\ldots,d} \right] C_{\Phi_k^T,\Phi_k^T}^{-1}, \quad \quad  k = 2,3,\ldots,d-1, \\
\frac{\partial \Psi_d}{\partial t} &= \left\langle \Psi_{d-1}^T \cdots \Psi_1^T, N(u) \right\rangle_{1,\ldots,d-1}.
\end{aligned}
\end{equation}
In these equations, $\langle \cdot, \cdot \rangle_{i_k,\ldots,i_{j}}
= \langle \cdot, \cdot \rangle_{L^2_{\mu_{i_k} \times \cdots \times \mu_{i_j}}(\Omega_{i_k} \times \cdots \times \Omega_{i_j})}$ and $\Phi_k$ is the multivariate FTT core 
in the $k$-th step of the FTT decomposition, i.e., the column 
vector that has components $\varphi_k(\alpha_k; x_{k+1},\ldots,x_d)$ ($\alpha_k=1,\ldots r_k$) -- see, e.g.,  Eq. \eqref{FFT_level_2}.
\end{theorem}
We prove  Theorem \ref{thm:evolution_equations} 
in \ref{sec:minimizing_lagrangian}. The PDE system 
\eqref{DO-FTT_system} will be referred to 
as dynamically orthogonal \cite{do} 
functional tensor train (DO-FTT) propagator .
As the solution evolves in time on the tensor manifold 
$\mathfrak{T}_{r}^{(d)}$, it is possible for some of the cores 
to become linearly dependent. As a consequence, the 
auto-correlation matrices $C_{\Phi_k^T,\Phi_k^T}$ 
become singular and the equations \eqref{DO-FTT_system} 
are no longer valid. In this case, the solution lives on a 
tensor manifold of smaller rank, say $\mathfrak{T}^{(d)}_{s}$, 
where $s_i \leq r_i$ for all $i = 1,2,\ldots,d$, and 
the dynamic tensor approximation can be constructed on $\mathfrak{T}^{(d)}_{s}$.
We conclude this section by emphasizing that it is possible 
to transform the dynamically orthogonal tensor cores 
$\Psi_i$ into bi-orthogonal cores (with corresponding 
bi-orthogonal equations) by adopting the proofs 
given in \cite{Dektor_2020,do/bo_equiv}. 
In light of the discussion above on the optimality of the dynamically orthogonal FTT integrator on $\mathfrak{T}_{r}^{(d)}$ 
and Lemma \ref{lemma:two_dd_decomps}, it is clear that 
FTT with bi-orthogonal cores is also an optimal\footnote{By selecting a collection of 
time-dependent invertible matrices $P_i(s) \in \text{GL}_{r_i \times r_i}(\mathbb{R})$, $i = 2,3,\ldots,d-1$, defined by the matrix differential equation 
\begin{equation}
\label{MDE}
\begin{cases}
\displaystyle \frac{d P_i(s)}{d s} = G_i(P_i)\vspace{0.1cm}\\
P_i(0) = P_{i,0}
\end{cases}
\end{equation} 
it is possible to develop evolution equations different 
than \eqref{DO-FTT_system} for $\Psi_i(s)$, 
which still solve the minimization problem 
\eqref{min_over_cores}. In fact all possible solutions 
to \eqref{min_over_cores} can be obtained in this 
way.} dynamic 
approximation on $\mathfrak{T}_{r}^{(d)}$.

\section{An application to the Fokker--Planck equation}
\label{sec:numerics}

In this section we demonstrate the 
dynamically orthogonal FTT integrator \eqref{DO-FTT_system} 
on a four-dimensional ($d=4$) 
Fokker--Planck equation with non-constant drift and 
diffusion coefficients. As is well known \cite{Risken},
the Fokker--Planck equation describes the evolution 
of the probability density function (PDF) of the state vector
solving the It\^o stochastic differential equation (SDE)
\begin{equation}
\label{Ito_SDE}
d X_t = \mu(X_t,t)dt + \sigma(X_t,t)dW_t.
\end{equation}
Here, $X_t$ is the $d$-dimensional state vector, $\mu(X_t,t)$ is the $d$-dimensional drift, $\sigma(X_t,t)$ is an $d \times m$ 
matrix and $W_t$ is an $m$-dimensional standard 
Wiener process. The Fokker--Planck equation that 
corresponds to \eqref{Ito_SDE} has the form 
\begin{equation}
\label{Fokker_Planck}
\frac{\partial p(x,t)}{\partial t} = \mathcal{L}(x,t) p(x,t), \qquad 
p(x,0) = p_0(x),
\end{equation}
where $p_0(x)$ is the PDF of the initial state $X_0$,
$\mathcal{L}$ is a second-order linear differential operator 
defined as
\begin{equation}
\mathcal{L}(x,t)p(x,t) = -\sum_{k=1}^d 
\frac{\partial }{\partial x_k}\left(\mu_k(x,t)p(x,t)\right)+ 
\sum_{k,j=1}^d \frac{\partial^2}{\partial x_k\partial x_j}
\left( D_{ij}(x,t)p(x,t)\right),
\label{L}
\end{equation}
and $D(x,t)=\sigma(x,t)\sigma(x,t)^T/2$ is the diffusion 
tensor. 
For our numerical demonstration we set 
\begin{equation}
\label{drift_diffusion}
\mu(x) = \alpha \begin{bmatrix}
\sin(x_1) \\
\sin(x_3) \\
\sin(x_4) \\
\sin(x_1)
\end{bmatrix}, \qquad
\sigma(x) = \sqrt{2 \beta} \begin{bmatrix}
g(x_2) & 0 & 0 & 0 \\
0 & g(x_3) & 0 & 0 \\
0 & 0 & g(x_4) & 0 \\
0 & 0 & 0 & g(x_1) \\ 
\end{bmatrix},
\end{equation}
where we define $g(x)=\sqrt{1 + k \sin(x)}$. 
With the drift and diffusion matrices 
chosen in \eqref{drift_diffusion} the 
operator \eqref{L} takes the form  
\begin{equation}
\label{Fokker_Planck_op}
\begin{aligned}
\mathcal{L} =& -\alpha \left(\cos(x_1) + \sin(x_1) \frac{\partial }{\partial x_1} + \sin(x_3) \frac{\partial }{\partial x_2} + \sin(x_4) \frac{\partial }{\partial x_3} + \sin(x_1) \frac{\partial }{\partial x_4} \right) \\
&+ \beta \left( (1 + k \sin(x_2) ) \frac{\partial^2 }{\partial x_1^2} 
+ (1 + k \sin(x_3) ) \frac{\partial^2 }{\partial x_2^2} + 
(1 + k \sin(x_4) ) \frac{\partial^2 }{\partial x_3^2} + 
(1 + k \sin(x_1) ) \frac{\partial^2 }{\partial x_4^2} \right).
\end{aligned}
\end{equation}
This is a linear, time-independent separable operator of 
rank $9$, since it can be written as
\begin{equation}
\label{separable_operator}
\mathcal{L} = \sum_{i=1}^9 L_i^{(1)} \otimes L_i^{(2)} 
\otimes L_i^{(3)} \otimes L_i^{(4)} ,
\end{equation}
where each $L_i^{(j)}$ operates on $x_j$ 
only. Specifically, we have  
\begin{equation}
\begin{array}{llll}
L_1^{(1)} = -\alpha \cos(x_1),  &\displaystyle L_2^{(1)} = -\alpha \sin(x_1) \frac{\partial}{\partial x_1},  & \displaystyle L_3^{(2)} = -\alpha \frac{\partial}{\partial x_2},  
&L_3^{(3)} = \sin(x_3), \vspace{0.1cm}\\
L_4^{(3)} =\displaystyle  -\alpha \frac{\partial}{\partial x_3},  
&L_4^{(4)} = \sin(x_4),  &L_5^{(1)} = -\alpha \sin(x_1),  &L_5^{(4)} = \displaystyle \frac{\partial}{\partial x_4}, 
\vspace{0.1cm}\\ 
L_6^{(1)} = \displaystyle \beta \frac{\partial^2}{\partial x_1^2}, & L_6^{(2)} = 1+k\sin(x_2),  &L_7^{(2)} = \displaystyle \beta \frac{\partial^2}{\partial x_2^2}, & L_7^{(3)} = 1+k\sin(x_3), 
\vspace{0.1cm}\\
L_8^{(3)} = \displaystyle \beta \frac{\partial^2}{\partial x_3^2},  
&L_8^{(2)} = 1+k\sin(x_4),  &L_9^{(4)} = \displaystyle \beta \frac{\partial^2}{\partial x_4^2},  &L_9^{(1)} = 1+k\sin(x_1),
\end{array}
\end{equation}
and all other unspecified $L_i^{(j)}$ are identity operators. 
We set the 
parameters in \eqref{drift_diffusion} as 
$\alpha = 0.1, \beta = 2.0, k = 1.0$ and 
consider the domain $\Omega = [0,2\pi]^4$ 
with periodic boundary conditions.
The initial PDF is set as 
\begin{equation}
\label{PDF_IC}
p_0(x) =\frac{\exp(\cos(x_1 + x_2 + x_3 + x_4))}{\|\exp(\cos(x_1 + x_2 + x_3 + x_4))\|_{L^1(\Omega)}}.
\end{equation}
To compute the FTT decomposition of $p_0(x)$ 
we first discretize it on a tensor product grid of 
$21$ evenly-spaced Fourier points in each variable $x_j$
($194481$ total points). The discrete tensor is then 
decomposed in the TT format using the 
TT-toolbox \cite{TT-toolbox} with 
appropriate quadrature weights (see \cite[\S 4.4]{Bigoni_2016}) 
and threshold set to $(21/2\pi)^4 \epsilon$.
In particular, we set $\epsilon =\{10^{-8}, 10^{-5}, 10^{-3}\}$ 
to obtain 
\begin{equation}
\label{FTT_IC}
\begin{aligned}
p_0(x,\epsilon) &= \sum_{\alpha_0,\ldots,\alpha_4=1}^{r(\epsilon)} \Psi_1(0) \Psi_2(0) \Psi_3(0) \Psi_4(0),
\end{aligned}
\end{equation}
with FTT ranks 
\begin{equation}
r(10^{-8}) = \begin{bmatrix}
1 \\
15 \\
15 \\
15 \\
1
\end{bmatrix}, \qquad 
r(10^{-5}) = \begin{bmatrix}
1 \\
9 \\
9 \\
9 \\
1
\end{bmatrix}, \qquad 
r(10^{-3}) = \begin{bmatrix}
1 \\
5 \\
5 \\
5 \\
1
\end{bmatrix}.
\label{FTTranks}
\end{equation}

To obtain a benchmark solution with which to compare 
the DO-FTT solution, 
the PDE \eqref{Fokker_Planck} with initial 
condition \eqref{FTT_IC} is solved on a full tensor 
product grid of points on the hypercube 
$\Omega = [0,2\pi]^4$ with $21$ evenly 
spaced points in each direction. 
Derivatives in the operator $\mathcal{L}$ are 
computed with pseudo-spectral differentiation 
matrices \cite{spectral_methods_book}, 
and the resulting semi-discrete approximation 
(ODE system) is integrated with explicit 
four stage fourth order Runge Kutta method 
using time step $\Delta t = 10^{-3}$. 
The numerical solution we obtained in this 
way is denoted by $p_{f}(x,t)$. In Figure \ref{fig:FP_solution_time_evolution} (middle row) 
we plot the two-dimensional marginal
\begin{equation}
\label{marg_pdf}
p(x_1,x_2,t) = \int_{0}^{2\pi} \int_{0}^{2 \pi} p(x_1,x_2,x_3,x_4,t) dx_3 dx_4
\end{equation}
at $t=0.1$, $t=0.5$ and $t=1$. 

\begin{figure}[t]
\hspace{0.2cm}
\centerline{\footnotesize\hspace{-0.2cm}$t = 0.1$ \hspace{4.6cm} $t = 0.5$  \hspace{4.5cm} $t = 1.0$ }

\hspace{0.1cm} 
\centerline{
	\rotatebox{90}{\hspace{1.8cm}\footnotesize DO-FTT }
		\includegraphics[width=0.34\textwidth]{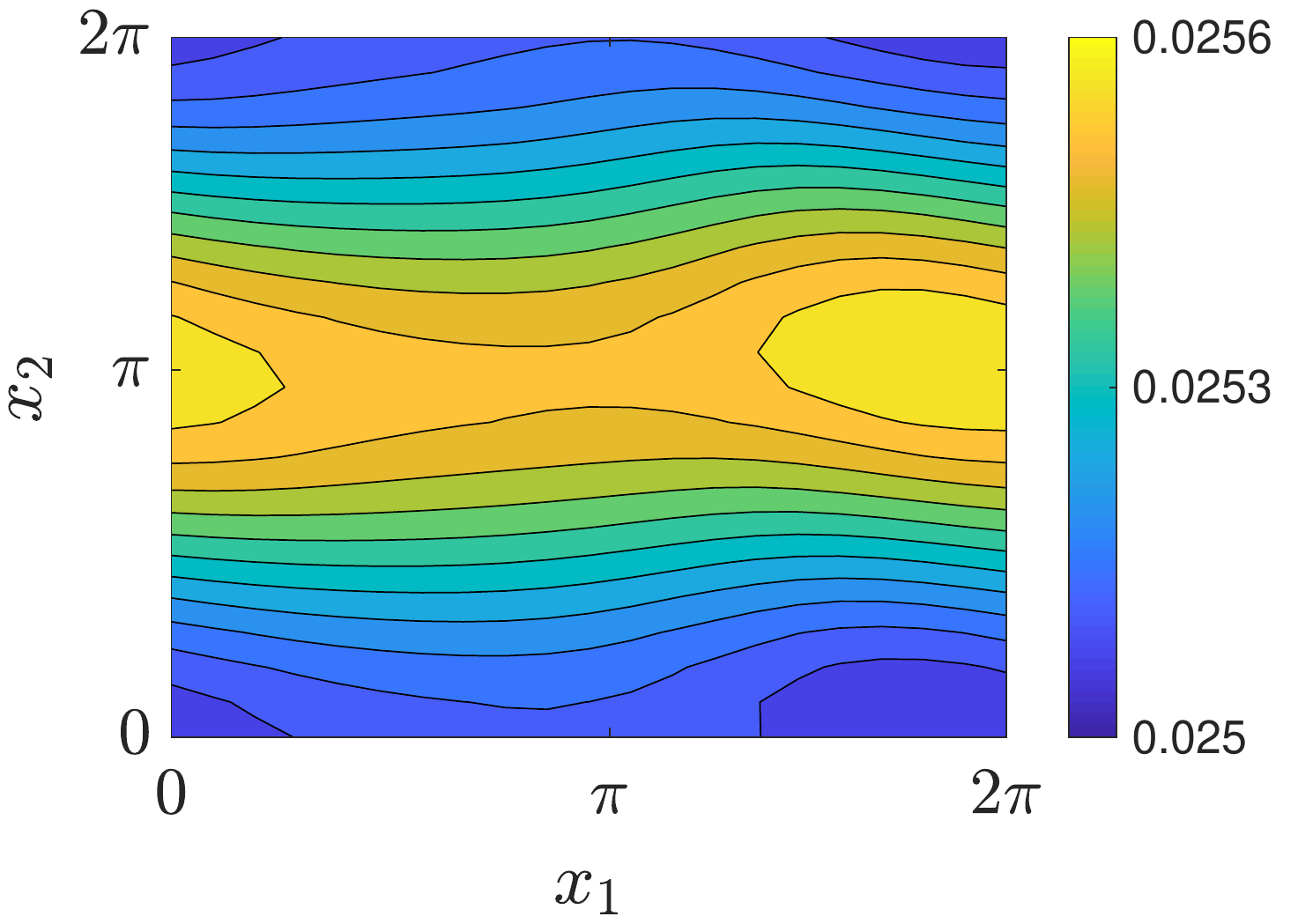}
		\includegraphics[width=0.34\textwidth]{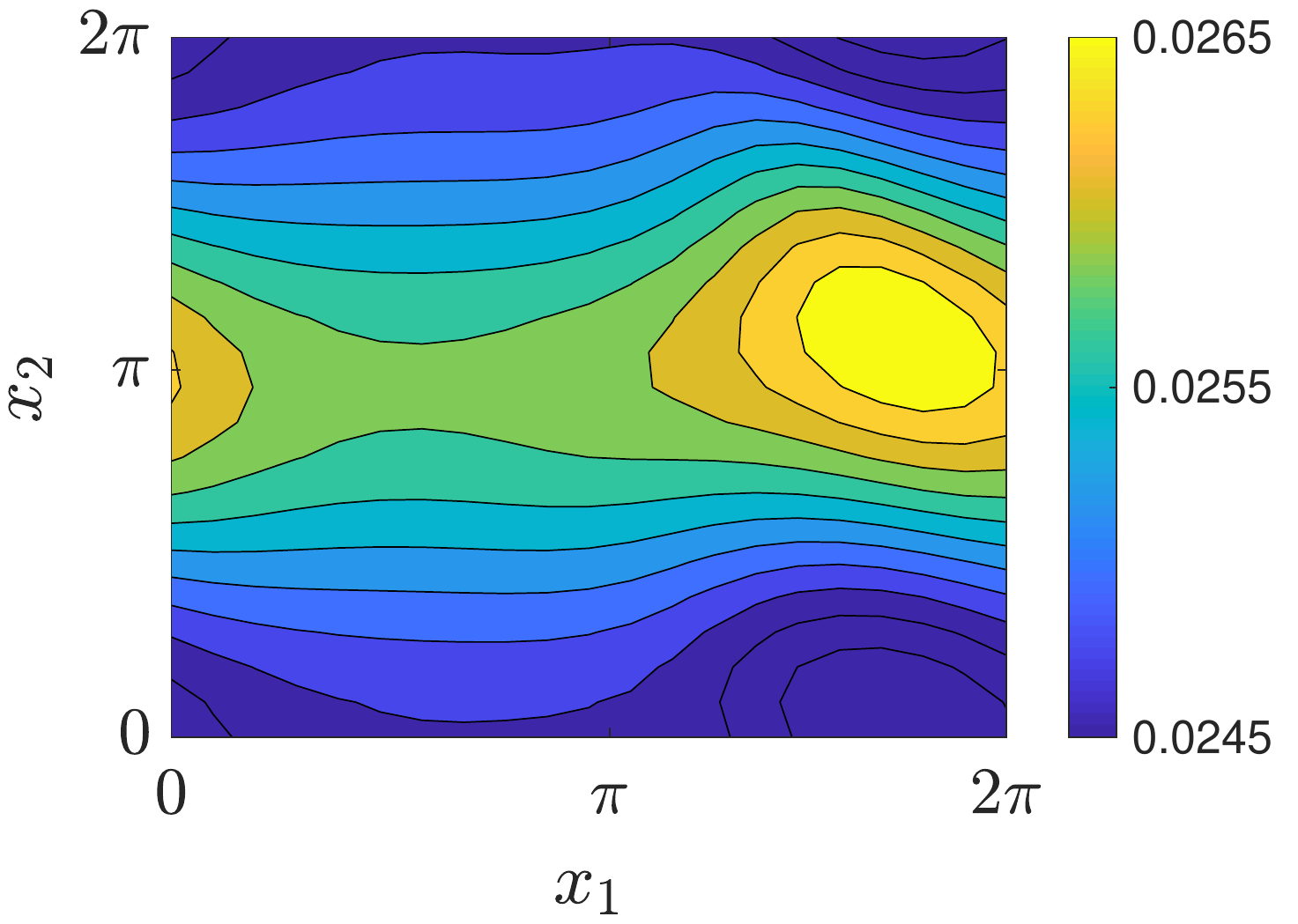}
		\includegraphics[width=0.34\textwidth]{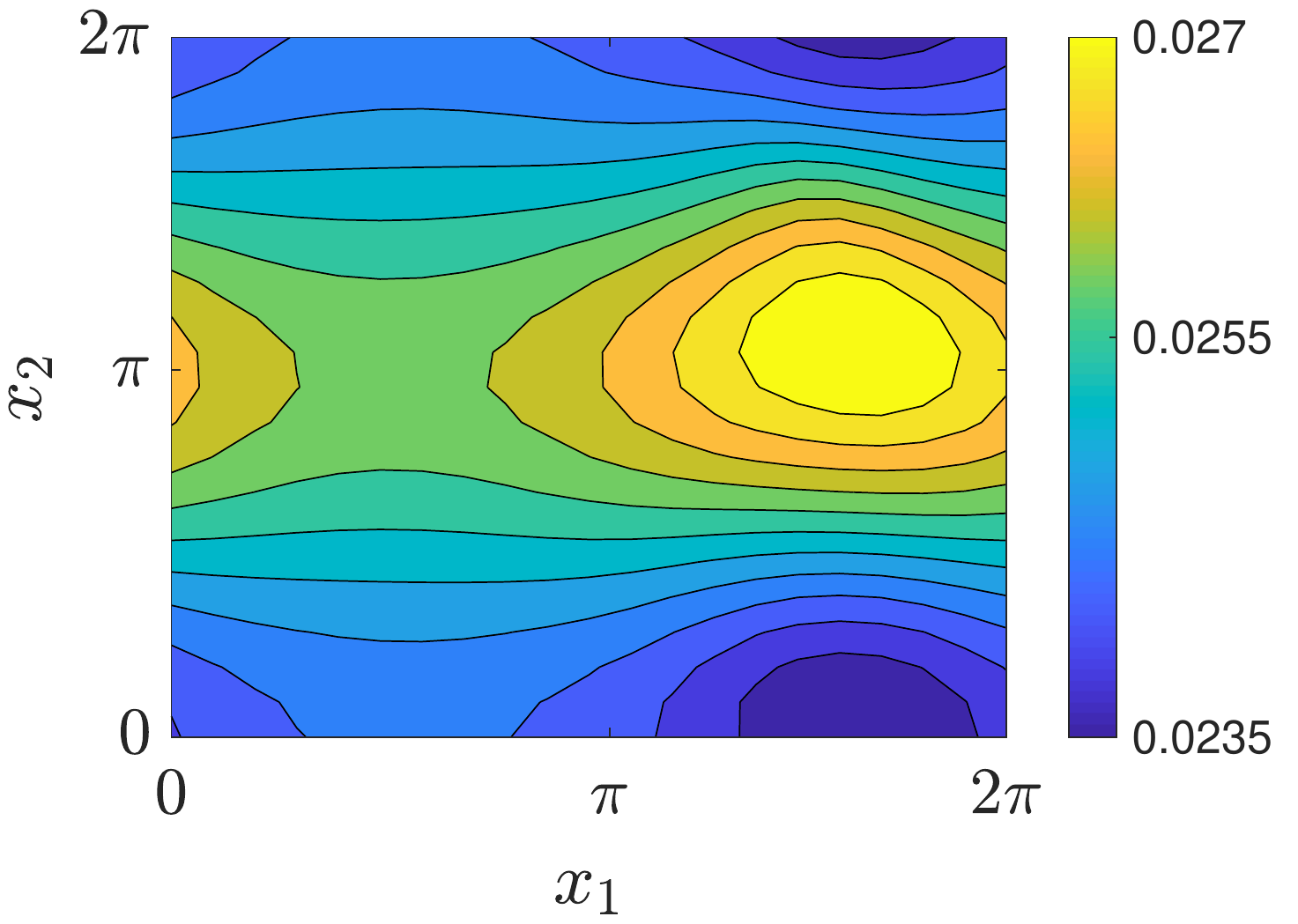}	
}
	
\hspace{0.1cm}	
\centerline{
\rotatebox{90}{\hspace{1.0cm}  \footnotesize  Full tensor product}
       \includegraphics[width=0.34\textwidth]{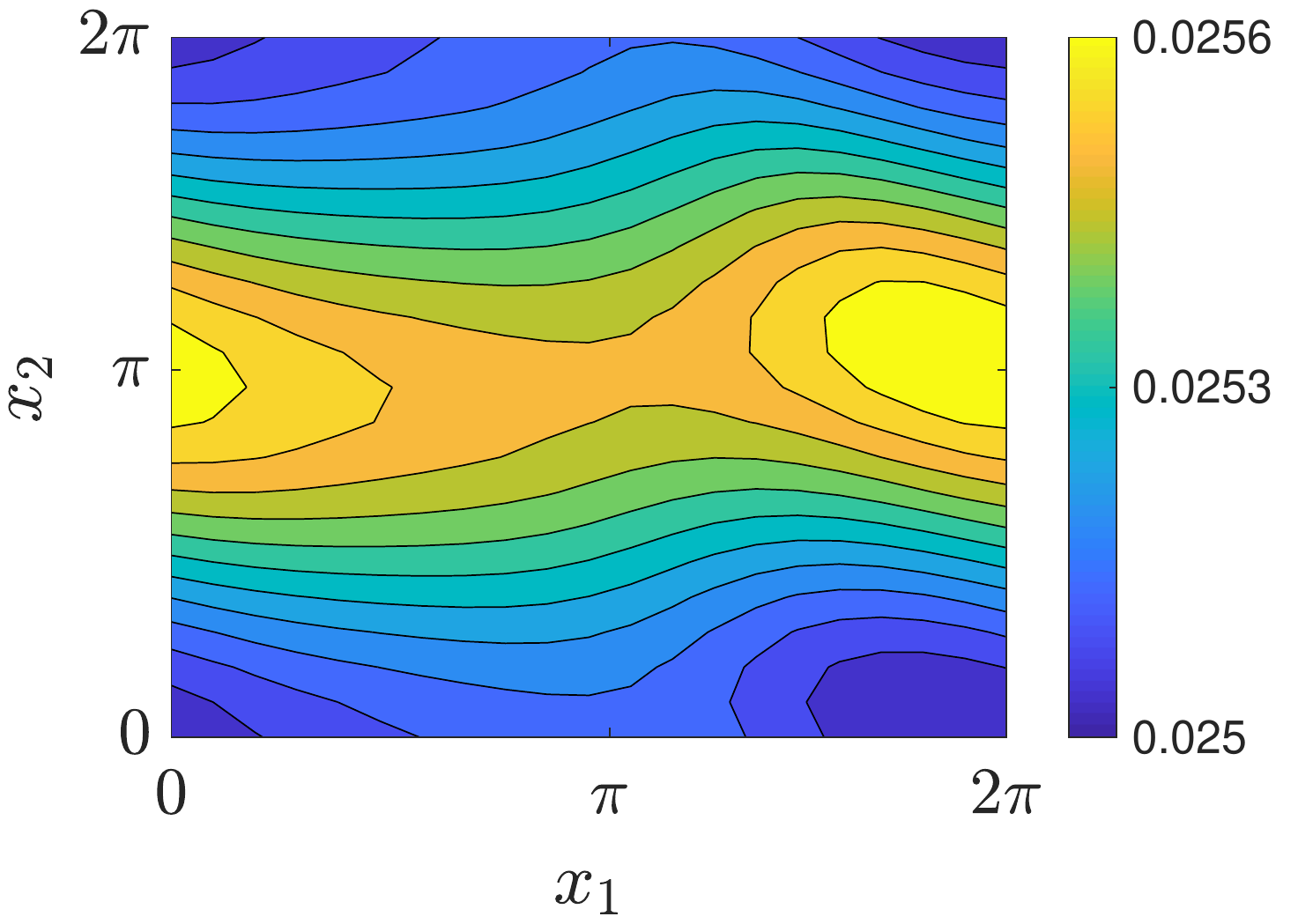}
	    \includegraphics[width=0.34\textwidth]{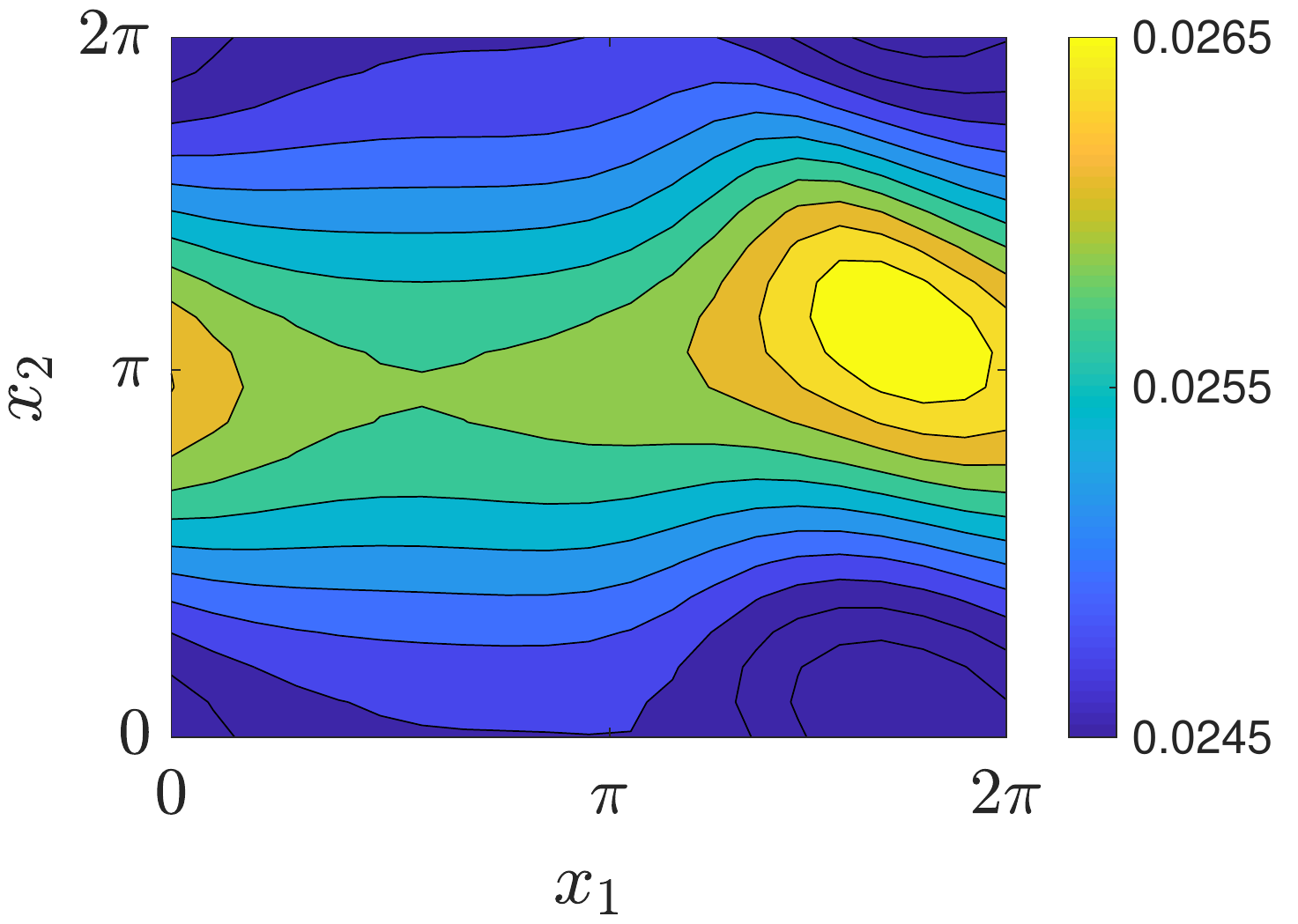}
	    \includegraphics[width=0.34\textwidth]{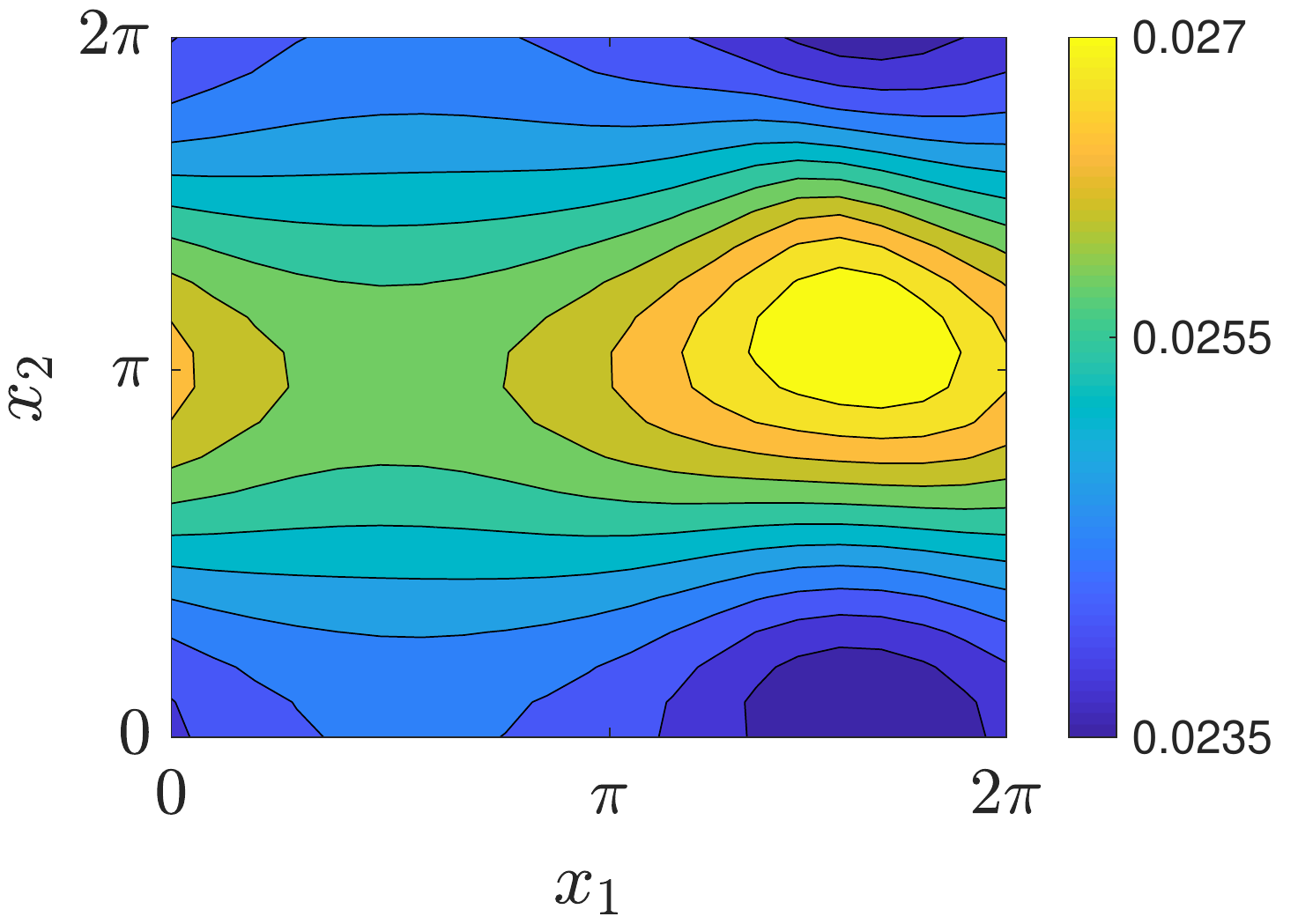}
}

\hspace{0.1cm}	
\centerline{
	\rotatebox{90}{\hspace{1.2cm}  \footnotesize Pointwise error}
       \includegraphics[width=0.34\textwidth]{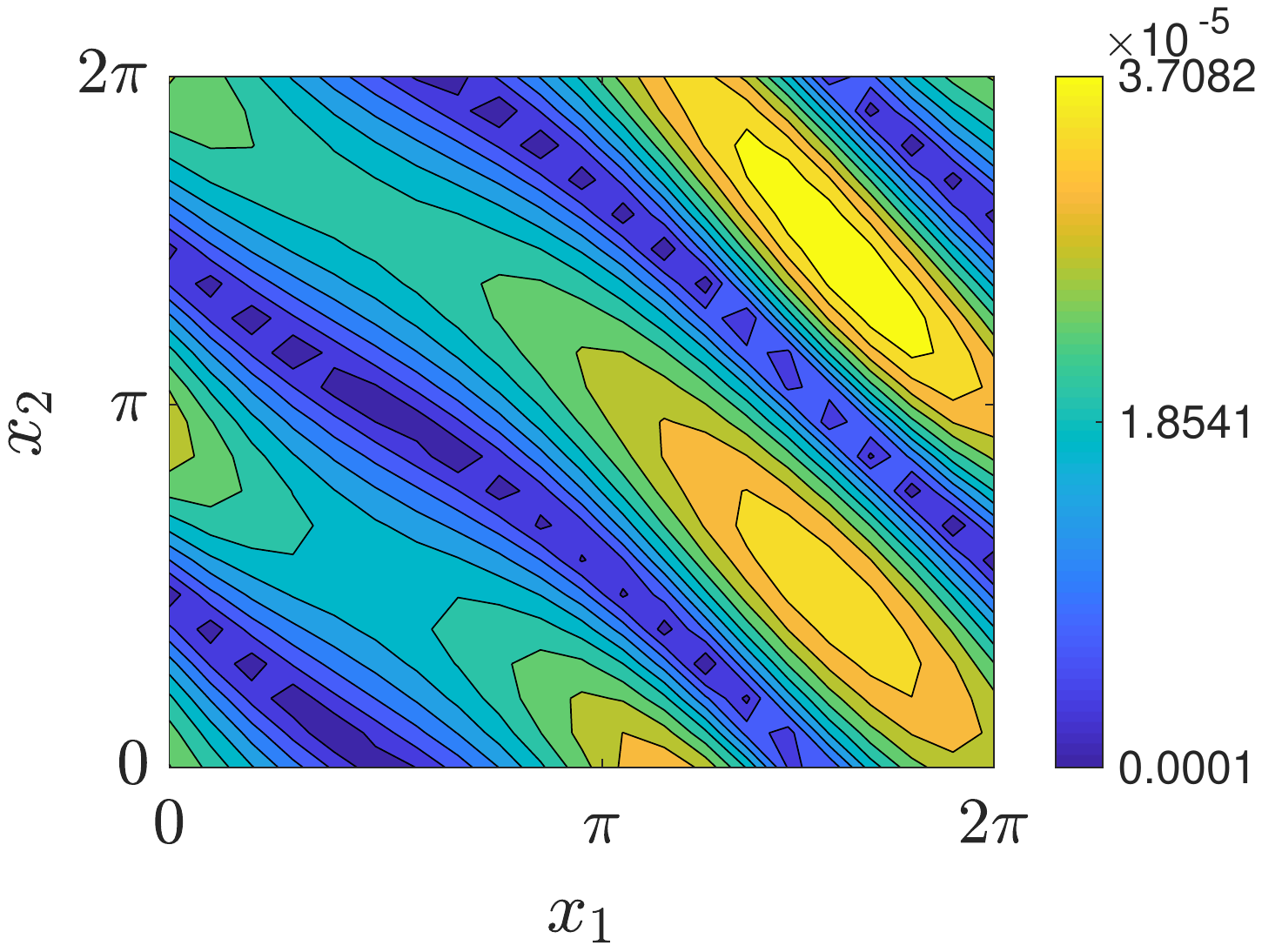}
	    \includegraphics[width=0.34\textwidth]{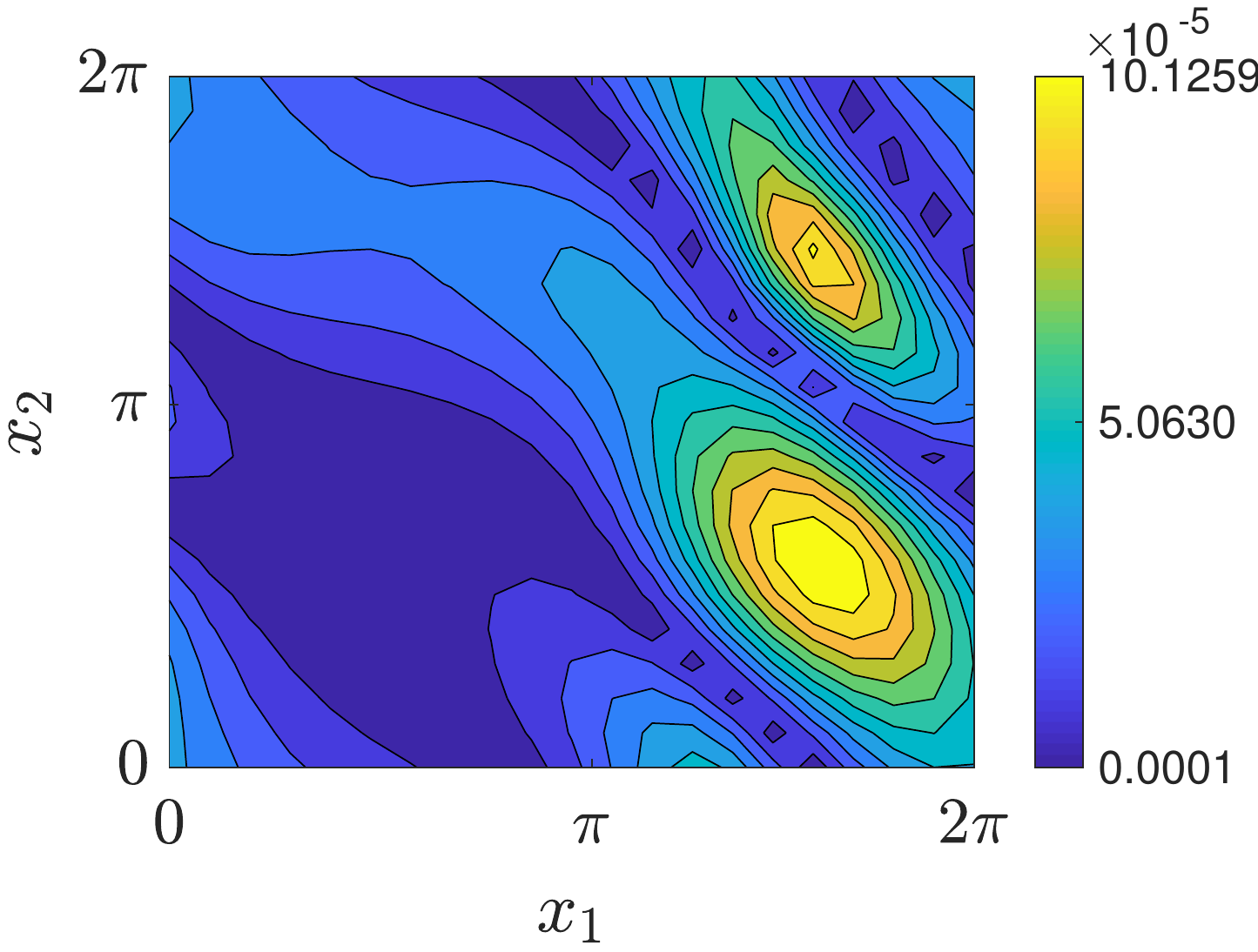}
	    \includegraphics[width=0.34\textwidth]{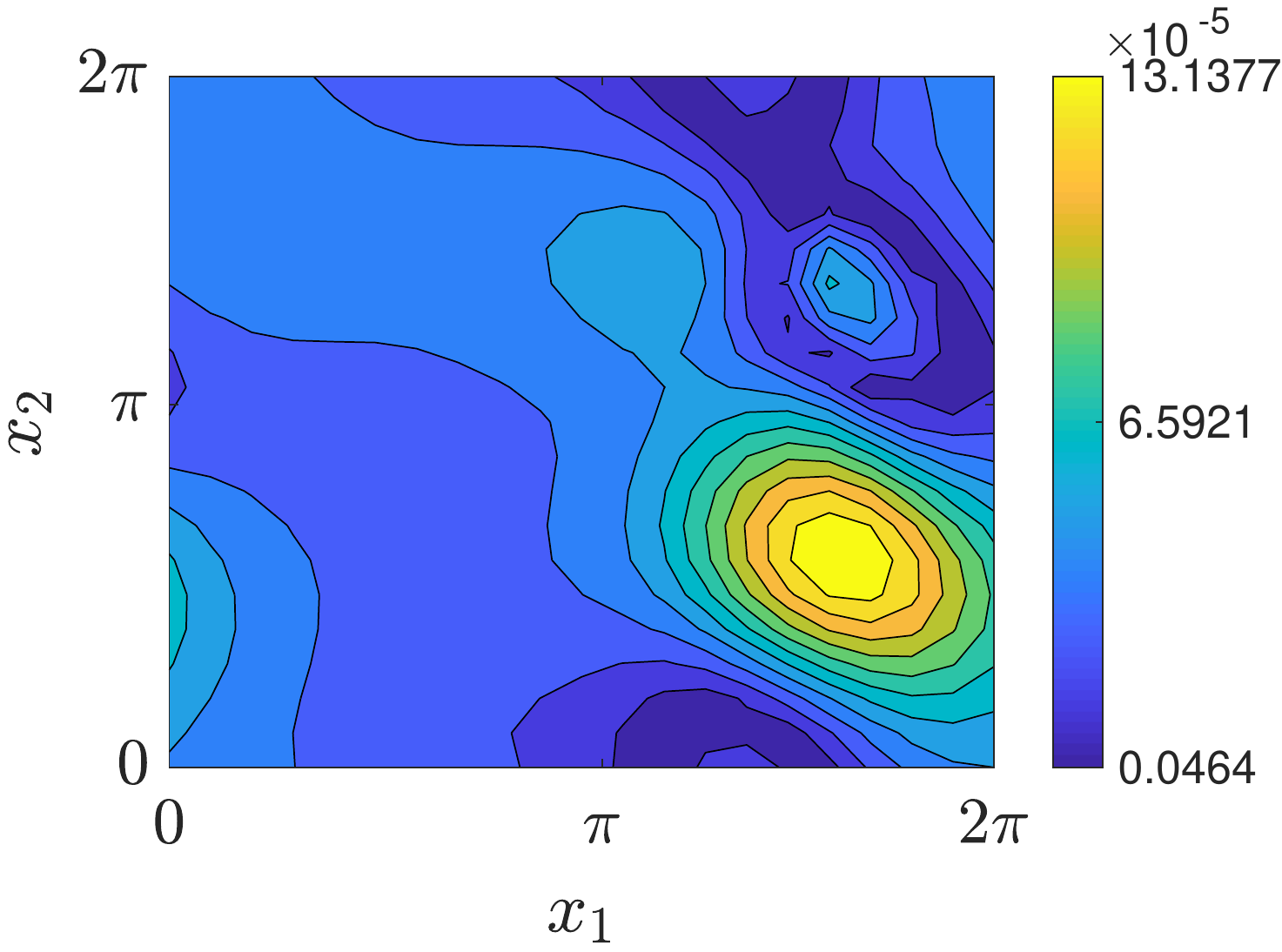}
}
\caption{Time snapshots of the marginal PDF $p(x_1,x_2,t)$ of the solution to \eqref{Fokker_Planck} computed with the DO-FTT propagator (top), on a full tensor product grid (middle) and and their pointwise error (bottom). The initial condition is obtained by decomposing \eqref{PDF_IC} as a TT-tensor with threshold 
$\epsilon = 10^{-8}$.}
\label{fig:FP_solution_time_evolution}
\end{figure}
\begin{figure}[t]
\centerline{\hspace{1.2cm}(a)\hspace{8.1cm}(b)}
	\centerline{
	\rotatebox{90}{\hspace{1.3cm} \footnotesize}
		\includegraphics[width=0.45\textwidth]{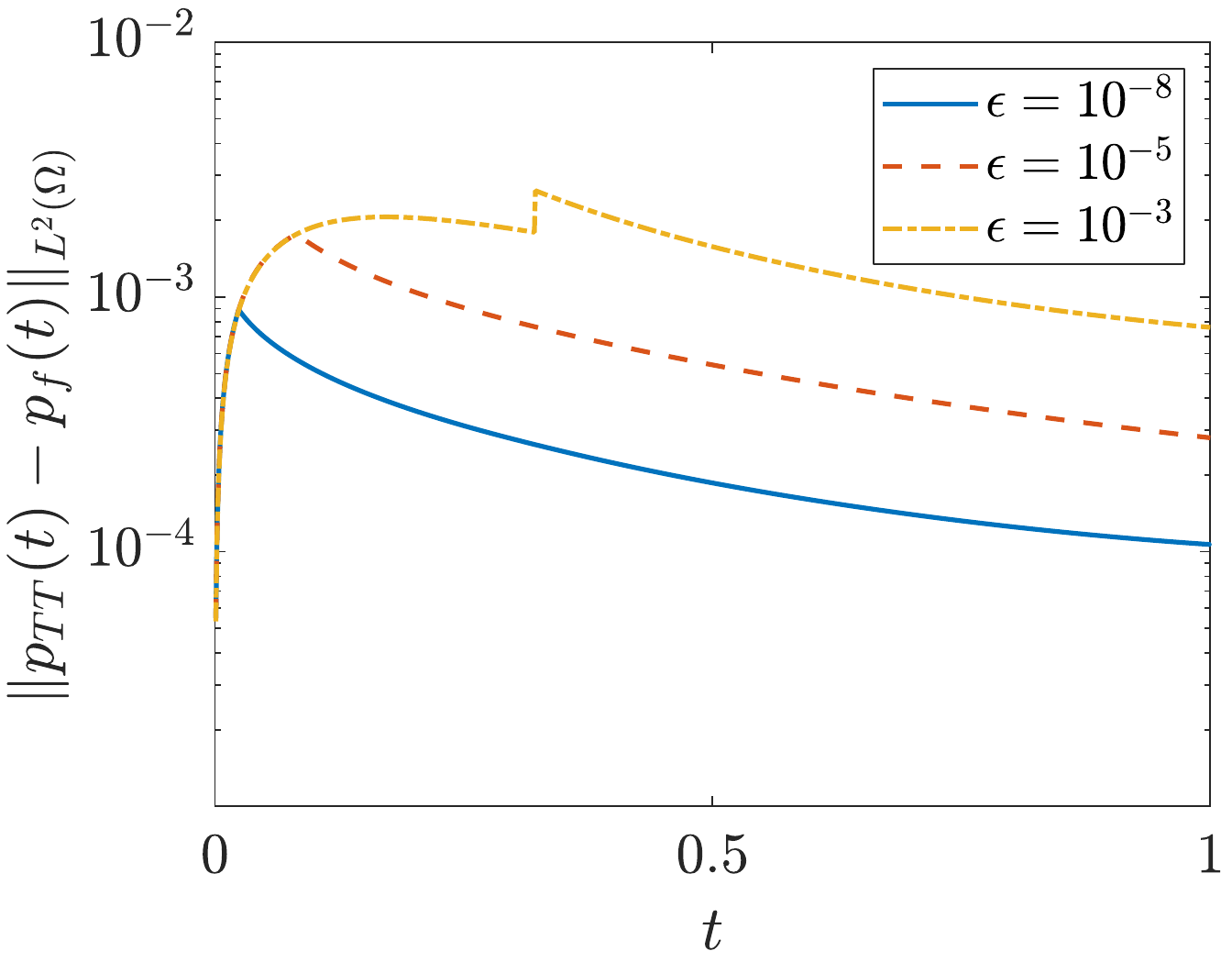}\hspace{1cm}
		\includegraphics[width=0.45\textwidth]{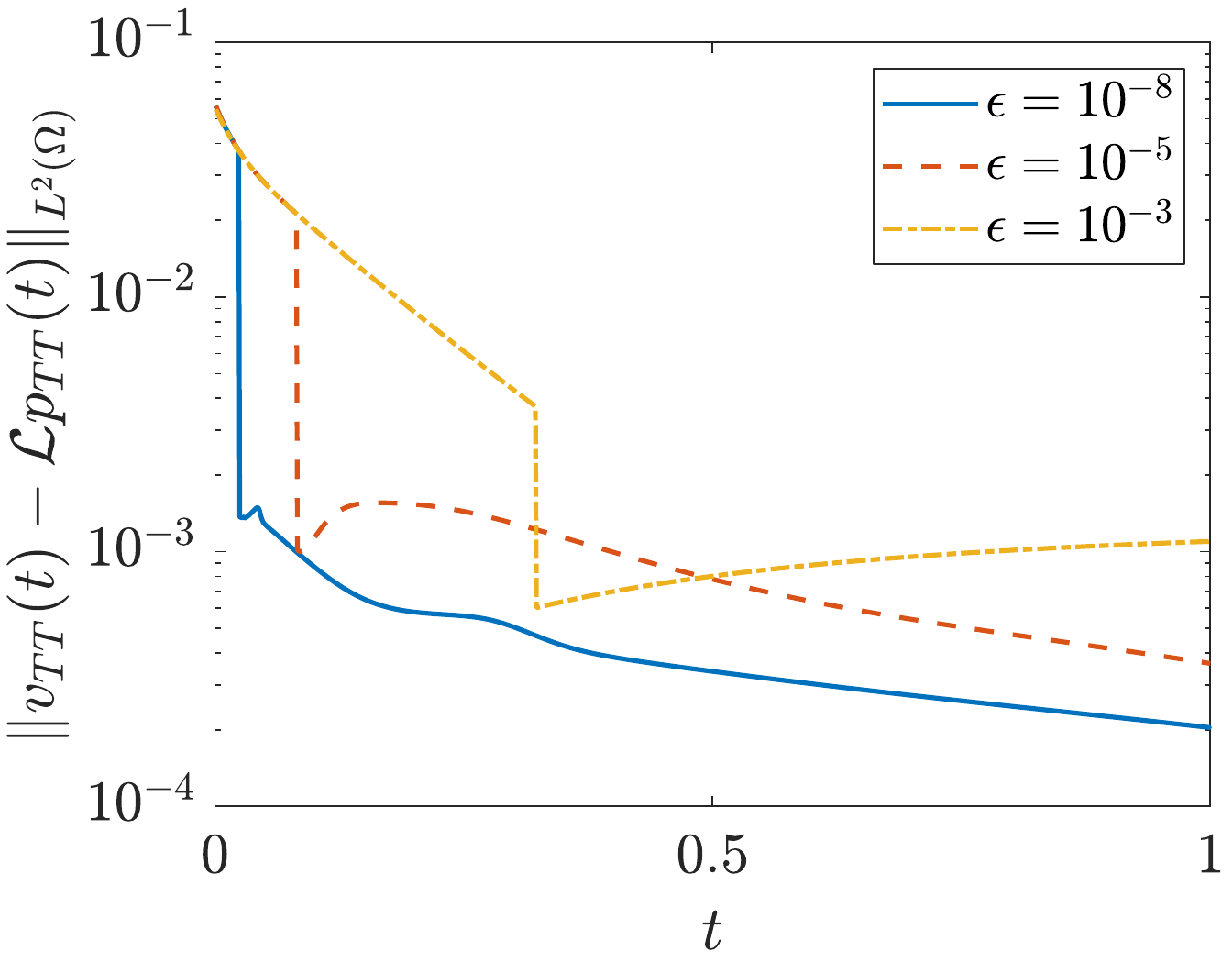}	
}
\caption{Error between the DO-FTT solution and the solution 
computed on a full tensor product grid to the Fokker--Planck 
equation \eqref{Fokker_Planck} (a). Error between the 
DO tangent vector to $\mathfrak{T}_{r}^{(4)}$ and 
$\mathcal{L} p_{\text{TT}}(t)$ (b). 
We plot results corresponding to FTT
decompositions of the initial condition with 
different thresholds $\epsilon$.}
\label{fig:error}
\end{figure}

\subsection{Fokker--Planck equation on the FTT tensor manifold} 
Next, we study the DO-FTT propagator \eqref{DO-FTT_system} 
for the Fokker--Planck equation \eqref{Fokker_Planck} 
with separable operator of the form \eqref{separable_operator}.
To write down such propagator explicitly, we 
adopt the convention that an operator applied 
to a matrix of functions acts on each entry of 
the matrix. With this convention, we obtain
the following evolution equations for the tensor cores
\begin{equation}
\label{FP_DO_FTT_system}
\begin{aligned}
\frac{\partial {\Psi}_1}{\partial t} =& \sum_{i=1}^{9} \left[ L_i^{(1)}\Psi_1 \left\langle 
L_i^{(2,3,4)} \Phi_1,\Phi_1^T \right\rangle_{2,3,4} - \Psi_1 \left\langle 
\Psi_1^T, L_i^{(1)} \Psi_1 \right\rangle_1 \left\langle L_i^{(2,3,4)} \Phi_1 , \Phi_1^T \right\rangle_{2,3,4} \right] C_{\Phi_1^T,\Phi_1^T}^{-1}, \\
\frac{\partial {\Psi}_2}{\partial t}  =& \sum_{i=1}^{9} \left[ \left\langle 
\Psi_1^T,L_i^{(1)} \Psi_1 \right\rangle_1 L_i^{(2)} \Psi_2 \left\langle L_i^{(3,4)} \Phi_2 , \Phi_2^T \right\rangle_{3,4} -\right.\\
&\left. \Psi_2 \left\langle 
\Psi_2^T \Psi_1^T,L_i^{(1)} \Psi_1 L_i^{(2)} \Psi_2 \right\rangle_{1,2} 
\left\langle L_i^{(3,4)} \Phi_2 , \Phi_2^T \right\rangle_{3,4} \right] C_{\Phi_2^T,\Phi_2^T}^{-1}, \\
\frac{\partial {\Psi}_3}{\partial t}  = &\sum_{i=1}^9 \biggr[ \left\langle \Psi_2^T \Psi_1^T, L_i^{(1)} \Psi_1 L_i^{(2)} \Psi_2 \right\rangle_{1,2} L_i^{(3)} \Psi_3 \left\langle L_i^{(4)} \Psi_4,\Psi_4 \right\rangle_4 \\
&- \Psi_3 \left\langle \Psi_3^T \Psi_2^T \Psi_1^T, L_i^{(1)} \Psi_1 L_i^{(2)} \Psi_2 L_i^{(3)} \Psi_3 \right\rangle_{1,2,3} \left\langle L_i^{(4)} \Psi_4, \Psi_4^T \right\rangle_4 \biggr]C_{\Psi_4^T,\Psi_4^T}^{-1} , \\
\frac{\partial {\Psi}_4}{\partial t}  = &\sum_{i=1}^{9} \left\langle \Psi_3^T \Psi_2^T \Psi_1^T, L_i^{(1)} \Psi_1 L_i^{(2)} \Psi_2 L_i^{(3)} \Psi_3 \right\rangle_{1,2,3} \Psi_4.
\end{aligned}
\end{equation}
This PDE system governs the dynamics of the solution to 
the Fokker--Planck equation \eqref{Fokker_Planck} with separable 
operator \eqref{separable_operator} on the FTT tensor 
manifold $\mathfrak{T}_r^{(4)}$. The rank $r$ can be adjusted 
adaptively in time \cite{Dektor_2020,robust_do/bo}, 
to guarantee a prescribed accuracy of the FTT solution.
\begin{figure}[t]
\centerline{\footnotesize\hspace{0.1cm} $\epsilon = 10^{-8}$ \hspace{4.2cm} $\epsilon = 10^{-5}$ \hspace{4.0cm} $\epsilon = 10^{-3}$}
	\centerline{
		\includegraphics[width=0.33\textwidth]{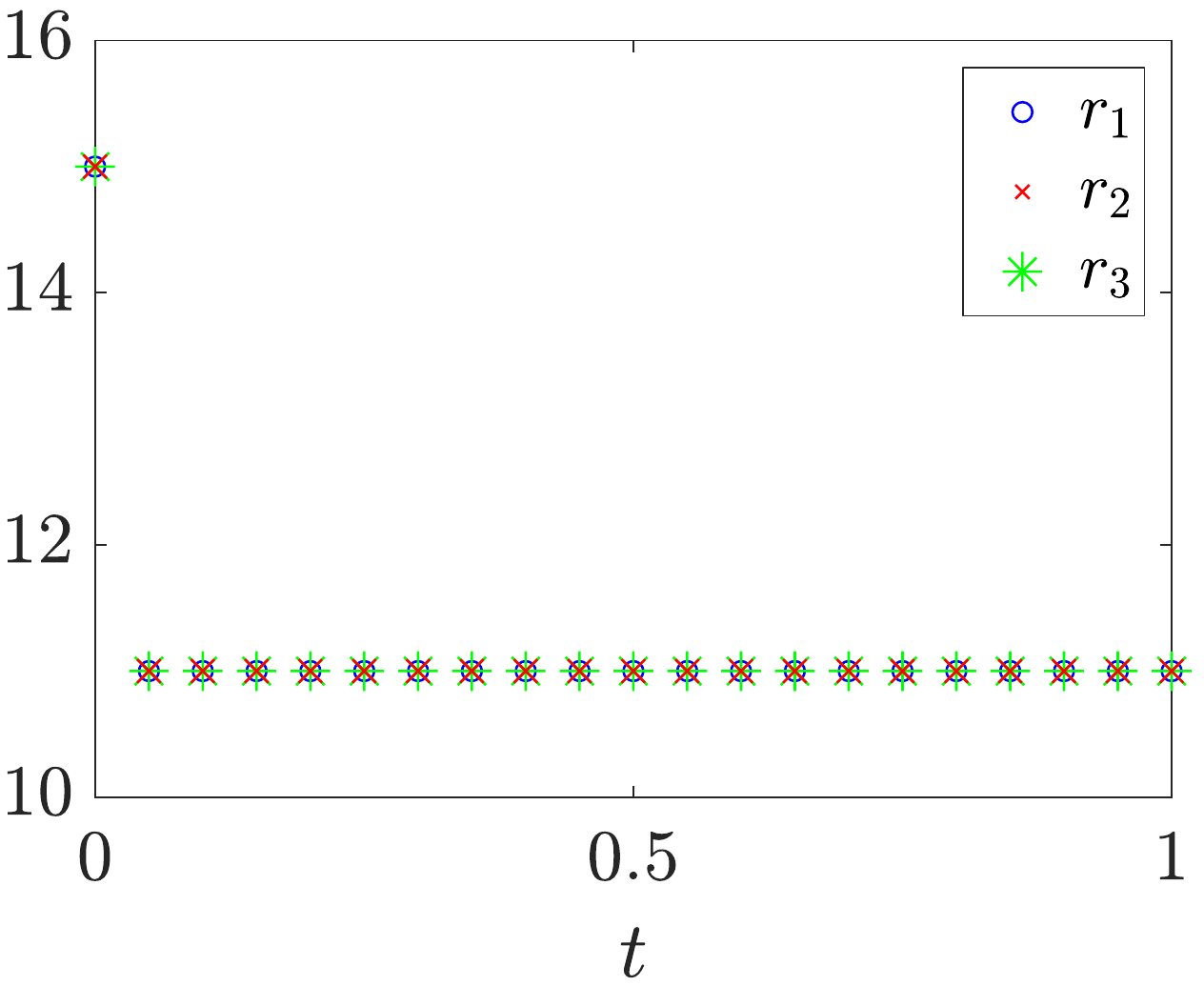}
		\includegraphics[width=0.33\textwidth]{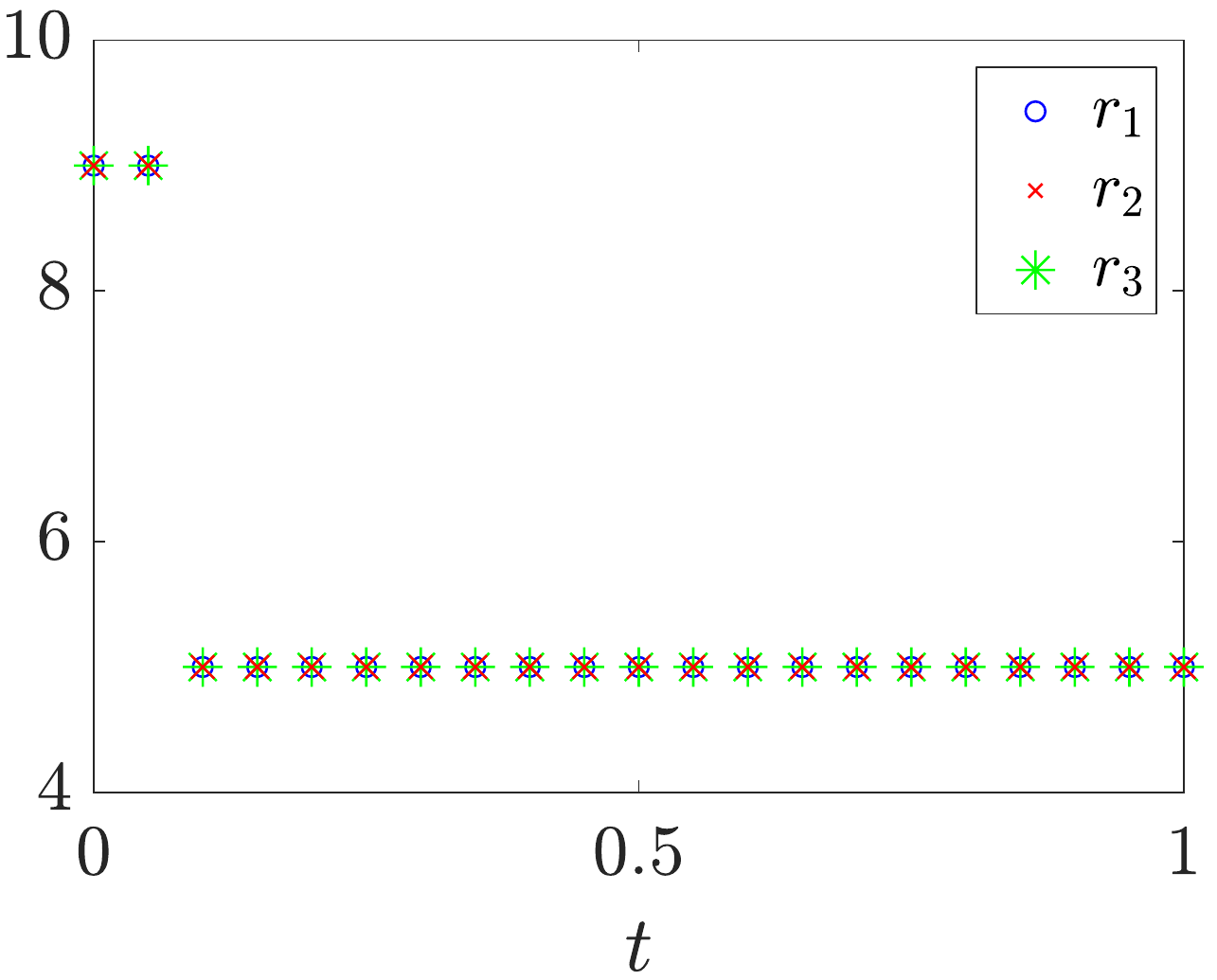}
		\includegraphics[width=0.33\textwidth]{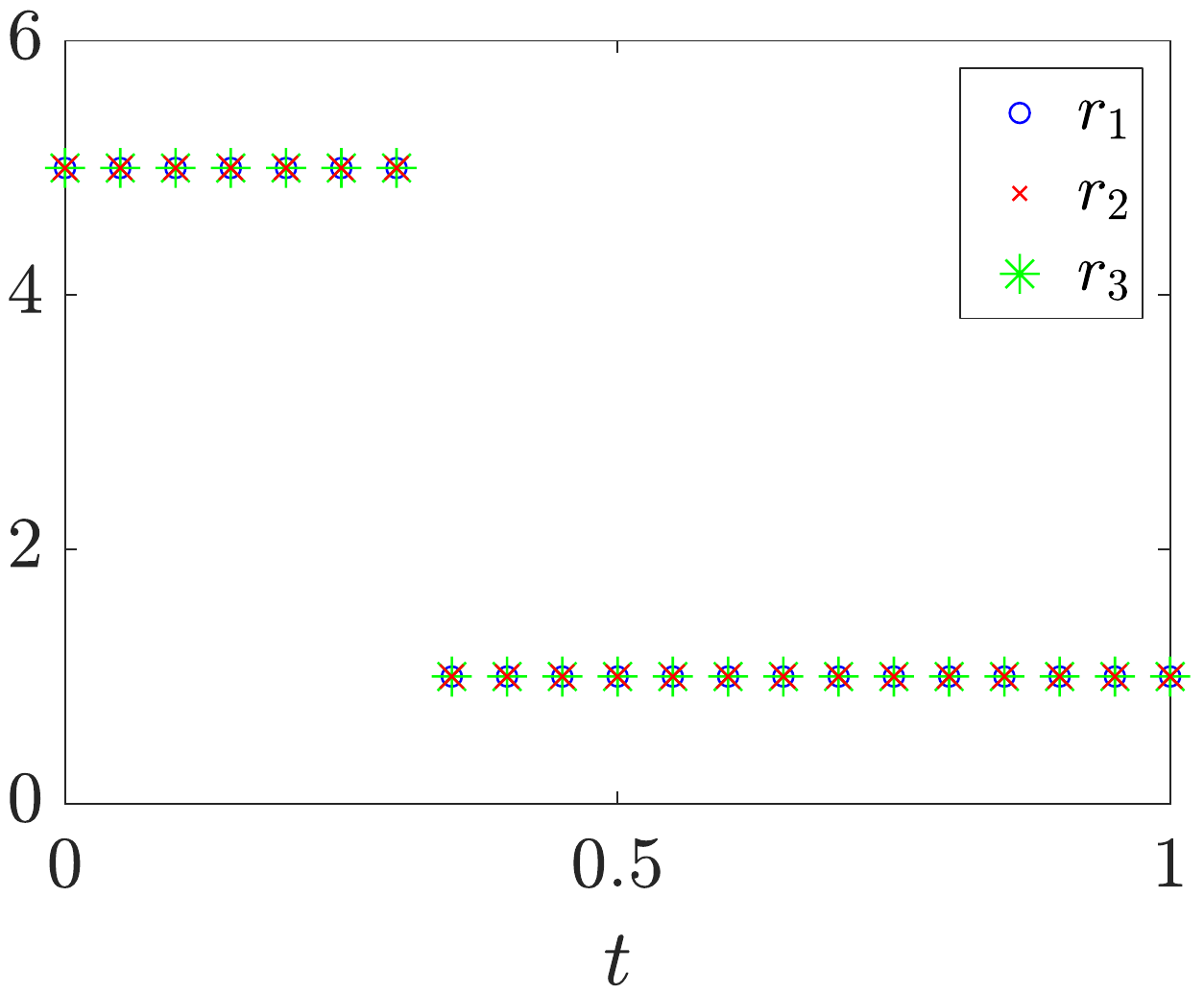}	
}
\caption{Time evolution of the DO-FTT solution rank for each simuation. Note that $r_0, r_4$ are excluded since they are always constantly equal to $1$.}
\label{fig:rank_evolution}
\end{figure}
The numerical solution to the PDE 
system \eqref{FP_DO_FTT_system} is 
computed with an explicit four--stages 
Runge-Kutta method with time 
step $\Delta t = 10^{-3}$, and a Fourier 
pseudo-spectral discretization \cite{spectral_methods_book}
on $21$ evenly--spaced collocation points in each 
spatial variable. 
In Figure \ref{fig:FP_solution_time_evolution} (top)
we plot a few temporal snapshots of the 
marginal PDF \eqref{marg_pdf} $p(x_1,x_2,t)$ we 
obtained using the DO-FTT temporal integrator.
The $L^2(\Omega)$ error between the benchmark solution 
and the DO-FTT solution is plotted in Figure \ref{fig:error}(a) 
for initial conditions decomposed with different thresholds 
$\epsilon$ (see Eqs. \eqref{FTT_IC}-\eqref{FTTranks}). 
With the velocity of each core $\partial \Psi_i/\partial t$ 
($i=1,2,3,4$) at time $t$ given by the DO-FTT system \eqref{FP_DO_FTT_system} we can construct the 
optimal tangent vector $v_{{TT}}(t)$ to 
the manifold $\mathfrak{T}^{(4)}_{r}$ 
at the point $p_{{TT}}(x,t)$ 
\begin{equation}
v_{{TT}}(t) = \frac{\partial \Psi_1}{\partial t} \Psi_2 \Psi_3 \Psi_4 + \Psi_1 \frac{\partial \Psi_2}{\partial t} \Psi_3 \Psi_4 + 
\Psi_1 \Psi_2 \frac{\partial \Psi_3}{\partial t} \Psi_4 + \Psi_1 \Psi_2 \Psi_3 \frac{\partial \Psi_4}{\partial t}.
\end{equation}
In Figure \ref{fig:error}(b) we plot the $L_{\mu}^2(\Omega)$ 
norm of $v_{{TT}}(t) - \mathcal{L} p_{{TT}}(t)$ 
at each time $t$. 

Note that the norm 
of $v_{{TT}}(t) - \mathcal{L} p_{{TT}}(t)$ 
is the norm of the normal component 
of $\mathcal{L} p_{{TT}}(t)$ at the point 
$p_{TT}(t)$ with respect to the manifold 
$\mathfrak{T}_r^{(4)}$ (see Eq. \eqref{tangent_and_normal}).
Such a norm measures the deviation 
between the temporal derivative of the DO-FTT solution 
$p_{TT}(x,t)$ and the temporal derivative defined by $\mathcal{L}p_{TT}(x,t)$ (right hand side of the Fokker--Planck equation). This provides an indication of whether the vector in the tangent plane 
of $\mathfrak{T}_r^{(4)}$ at $p_{TT}(x,t)$ is pointing in the 
right direction, and if the rank $r=(r_1,r_2,r_3,r_4)$ is sufficient
to resolve the dynamics. Note that the rank is 
initially set by $\epsilon$ (see Eq. \eqref{FTTranks}).
As $p_{{TT}}(x,t)$ propagates forward in time the 
energy of FTT modes (tensor cores) decays due 
to the diffusion term in the Fokker--Planck equation. 
If no action is taken to reduce solution rank,
low energy modes will lead to ill-conditioned 
(possibly singular) matrices 
$C_{\Phi_j^T,\Phi_j^T}$ resulting in instabilities 
of the DO-FTT propagator \eqref{FP_DO_FTT_system}. 
To ensure this does not happen 
the energy of each FTT mode is tracked and if the 
energy of one mode falls below the threshold $\epsilon$ 
then the FTT decomposition is recomputed with 
threshold $\epsilon$. For each of the three simulations 
we run, $\epsilon$ is kept constant throughout the 
integrating period $t\in [0,1]$ and set 
at $\epsilon = \{10^{-8},10^{-5},10^{-3}\}$. 
In Figure \ref{fig:rank_evolution} we plot the time 
evolution of the solution ranks we obtained 
for each of the three simulations.

\section{Summary}
\label{sec:summary}

We developed a new method based on functional 
tensor decomposition and dynamic 
tensor approximation to compute the solution of 
high-dimensional time-dependent nonlinear 
PDEs in real separable Hilbert spaces. 
The method is built upon the 
functional tensor train (FTT) expansion 
proposed by Bigoni {\em et al.} in \cite{Bigoni_2016}, 
combined with dynamic tensor approximation. 
This yields an infinite-dimensional 
analogue of the dynamic low--rank approximation on Euclidean 
manifolds studied by Lubich {\em et al.} for 
matrices \cite{Lubich_2007,Lubich_2008}, and for
hierarchical tensors \cite{Lubich_2013,Lubich_2015}. 
The idea of dynamic approximation is to project the 
time derivative of the PDE solution onto the 
tangent space of a low-rank functional tensor manifold 
at each time. Using the set of hierarchical 
dynamic orthogonality constraints we recently 
introduced in \cite{Dektor_2020} we computed the 
projection needed for dynamic approximation 
by minimizing a convex energy functional over the 
tangent space. The unique optimal velocity vector 
obtained in this way allows us to integrate the PDE forward in 
time on a tensor manifold of constant rank. 
In the case of initial/boundary value problems defined 
in separable geometries, this procedure yields evolution 
equations for the tensor modes in the form of a coupled 
system of one-dimensional time-dependent PDEs. 
We applied the proposed tensor method 
to a four-dimensional Fokker--Planck equation 
with non-constant drift and diffusion coefficients, 
and demonstrated its accuracy in predicting 
relaxation to statistical equilibrium.

\vspace{0.5cm}
\noindent 
{\bf Acknowledgements} 
This research was supported by the U.S. Army 
Research Office grant W911NF1810309, 
and by the U.S. Air Force Office of Scientific 
Research grant  FA9550-20-1-0174.

\appendix
\section{Proof of Theorem \ref{thm:evolution_equations}}
\label{sec:minimizing_lagrangian}

The functional $\mathcal{A}$ in 
\eqref{constrained_min_problem_lagrangian} 
is convex and thus a critical point is necessarily 
a global minimum. To find such a critical point set 
the first variation of $\mathcal{A}$ 
with respect to $\partial \psi_j(\xi_{j-1},\xi_j)/\partial s$ 
in the direction 
$\eta_j(\xi_{j-1},\xi_j)$ ($j = 1,2,\ldots,d$, $\xi_j = 1,2,\ldots,r_j$) 
\begin{equation}
\left[\frac{d}{d \epsilon} \mathcal{A}\left( \frac{\partial \psi_j(\xi_{j-1},\xi_j)}{\partial s} + \epsilon \eta_j(\xi_{j-1},\xi_j) \right) \right]_{\epsilon = 0}
\end{equation}
equal to zero for all $\eta_j(\xi_{j-1},\xi) \in L^2_{\mu_j}(\Omega_j)$.
Note that we have available the dynamic constraints 
\begin{equation}
\left\langle \frac{\partial \psi_j(\cdot,\alpha_j)}{\partial s}, \psi_j(\cdot,\beta_j) \right\rangle_{L^2_{\tau \times \mu_j}(\mathbb{N} \times \Omega_j)} = 0 , \qquad \forall j = 1,\ldots,d-1, \quad \alpha_j,\beta_j = 1,\ldots,r_j,
\end{equation}
and the static constraints 
\begin{equation}
\left\langle \psi_j(\cdot,\alpha_j), \psi_j(\cdot,\beta_j) \right\rangle_{L^2_{\tau \times \mu_j}(\mathbb{N} \times \Omega_j)} = \delta_{\alpha_j,\beta_j} , \qquad \forall j = 1,\ldots,d-1, \quad \alpha_j,\beta_j = 1,\ldots,r_j,
\end{equation}
which are implied by the dynamic constraints as long as the cores $\Psi_1(s),\ldots,\Psi_{d-1}(s)$ all have identity auto-correlation matrices at some time (say at $s=0$).
For $j = 1$ we obtain 
\begin{equation}
\begin{aligned}
&\left[\delta_{\frac{\partial {\psi}_1(\xi_{0},\xi_1)}{\partial s}} \mathcal{A}\right] \eta_1 (\xi_{0} , \xi_1) 
\\
&= 2\biggl\langle  \frac{\partial }{\partial s} \left[ \sum_{\alpha_0,\alpha_1=1}^{r_0,r_1} \psi_1(\alpha_0,\alpha_1) \varphi_1(\alpha_1) \right] - N(u) , 
\eta_1(\xi_0,\xi_1) \varphi_1(\xi_1) \biggr\rangle_{1,2,\ldots,d} \\
&+ \sum_{\alpha_1=1}^{r_j} \lambda_{\xi_1 \alpha_1}^{(1)} \bigg\langle
\eta_1(\xi_{0},\xi_1), \psi_1(\xi_{0},\alpha_1) \bigg\rangle_1 \\
&= 0, 
\end{aligned}
\end{equation}
whence the fundamental lemma of calculus of variations implies 
\begin{equation}
\begin{aligned}
2\biggl\langle  \frac{\partial }{\partial s} \left[ \sum_{\alpha_0,\alpha_1=1}^{r_0,r_1} \psi_1(\alpha_0,\alpha_1) \varphi_1(\alpha_1) \right] - N(u) , 
 \varphi_1(\xi_1) \biggr\rangle_{2,\ldots,d} 
+ \sum_{\alpha_1=1}^{r_j} \lambda_{\xi_1 \alpha_1}^{(1)} 
 \psi_1(\xi_{0},\alpha_1) 
= 0.
\end{aligned}
\end{equation}
Rearranging terms we obtain 
\begin{equation}
\label{psi_1_derivation}
\begin{aligned}
&\sum_{\alpha_0,\alpha_1=1}^{r_0,r_1} \left( \frac{\partial \psi_1(\alpha_0,\alpha_1)}{\partial s} \left\langle \varphi_1(\alpha_1),\varphi_1(\xi_1) \right\rangle_{2,\ldots,d} + \psi_1(\alpha_0,\alpha_1) \left\langle \frac{\partial \varphi_1(\alpha_1)}{\partial s}, \varphi_1(\xi_1) \right\rangle_{2,\ldots,d} \right) \\
&= \left\langle N(u),\varphi_1(\xi_1) \right\rangle_{2,\ldots,d} - \frac{1}{2} \sum_{\alpha_1=1}^{r_j} \lambda_{\xi_1 \alpha_1}^{(1)} 
 \psi_1(\xi_{0},\alpha_1) .
\end{aligned}
\end{equation}
Taking $\langle \cdot, \psi_1(\alpha_{0},\xi'_1) \rangle_{L^2_{\tau \times \mu_1}(\mathbb{N} \times \Omega_1)}$ of the previous equation and utilizing the dynamic 
and static constraints, 
we solve for the Lagrange multiplier
\begin{equation}
\label{multiplier_1}
\begin{aligned}
\lambda_{\xi_1 \xi'_1}^{(1)} = &\bigg\langle N(u), 
\psi_1(1,\xi'_1) \varphi_1(\xi_1) \bigg\rangle_{1,\ldots,d} - \bigg\langle \frac{\partial {\varphi}_1(\xi'_1)}{\partial s}, \varphi_1(\xi_1) \bigg\rangle_{2,\ldots,d}.
\end{aligned}
\end{equation}
Substituting \eqref{multiplier_1} into \eqref{psi_1_derivation} 
and rearranging terms we obtain
\begin{equation}
\begin{aligned}
&\sum_{\alpha_1=1}^{r_1} \frac{\partial \psi_1(1,\alpha_1)}{\partial s} 
\left\langle \varphi_1(\alpha_1),\varphi_1(\xi_1) \right\rangle_{2,\ldots,d} \\&= 
\left\langle N(u),\varphi_1(\xi_1) \right\rangle_{2,\ldots,d} 
- \sum_{\alpha_1=1}^{r_1} \psi_1(1,\alpha_1) \left\langle N(u), \psi_1(1,\alpha_1) \varphi_1(\xi_1) \right\rangle_{1,\ldots,d}
\end{aligned}
\end{equation}
Using the matrix--vector notation for tensor cores and 
inverting the auto-correlation matrix on the left hand side yields 
the equation for $\partial \Psi_1 / \partial s$ in \eqref{DO-FTT_system}.
For $j = 2, \ldots, d-1$ we have that 
\begin{equation}
\begin{aligned}
\left[\delta_{\frac{\partial {\psi}_j(\xi_{j-1},\xi_j)}{\partial s}} \mathcal{A}\right] \eta_j (\xi_{j-1} , \xi_j) &= 2
\biggl\langle  \frac{\partial }{\partial s} \left[ \sum_{\alpha_0,\ldots,\alpha_j=1}^{r_0,\ldots,r_j} \psi_1(\alpha_0,\alpha_1) \cdots 
\psi_j(\alpha_{j-1},\alpha_j) \varphi_j(\alpha_j) \right] - N(u) ,  \\
& \sum_{\alpha_0,\ldots, \alpha_{j-2}=1}^{r_0,\ldots,r_{j-2}} 
\psi_1(\alpha_0,\alpha_1) \cdots \psi_{j-1}(\alpha_{j-2},\xi_{j-1}) 
\eta_j(\xi_{j-1},\xi_j) \varphi_j(\xi_j) \biggr\rangle_{1,2,\ldots,d} \\
&+ \sum_{\alpha_j=1}^{r_j} \lambda_{\xi_j \alpha_j}^{(j)} \bigg\langle
\eta_j(\xi_{j-1},\xi_j), \psi_j(\xi_{j-1},\alpha_j) \bigg\rangle_j.
\end{aligned}
\end{equation}
Moreover, utilizing the fundamental lemma of 
calculus of variations and 
rearranging terms we obtain 
\begin{equation}
\begin{aligned}
&\bigg\langle \sum_{\alpha_0,\ldots,\alpha_j=1}^{r_0,\ldots,r_j} \frac{\partial }{\partial s} \bigg[ \psi_1(\alpha_0,\alpha_1) \cdots \psi_j(\alpha_{j-1},\alpha_j) \varphi_j(\alpha_j)\bigg] , \\ 
&\sum_{\alpha_0,\ldots,\alpha_{j-2} = 1 }^{r_0,\ldots,r_{j-2}} \psi_1(\alpha_0,\alpha_1) \cdots \psi_{j-1}(\alpha_{j-2},\xi_{j-1}) \varphi_j(\xi_j) \bigg\rangle_{1,\ldots,j-1,j+1,\ldots,d} \\
&= \bigg\langle N(u) , \sum_{\alpha_0,\ldots,\alpha_{j-2} = 1 }^{r_0,\ldots,r_{j-2}} \psi_1(\alpha_0,\alpha_1) \cdots \psi_{j-1}(\alpha_{j-2},\xi_{j-1}) \varphi_j(\xi_j) \bigg\rangle_{1,\ldots,j-1,j+1,\ldots,d} \\ &
- \frac{1}{2}\sum_{\alpha_j} \lambda_{\xi_j \alpha_j}^{(j)} \psi_j(\xi_{j-1},\alpha_j).
\end{aligned}
\label{A2}
\end{equation}
Utilizing the dynamic orthogonality condition \eqref{DO_constraints_multilevel} and the 
orthonormality for all $t$ on the left hand side of \eqref{A2}
we obtain 
\begin{equation}
\label{DO_derivation_1}
\begin{aligned}
&\sum_{\alpha_j} \left( \frac{{\psi}_j(\xi_{j-1},\alpha_j)}{\partial s} \bigg\langle \varphi_j(\alpha_j), \varphi_j(\xi_j) \bigg\rangle_{j+1,\ldots,d} + \psi_j(\xi_{j-1},\alpha_j) \bigg\langle \frac{\partial {\varphi}_j(\alpha_j)}{\partial s}, \varphi_j(\xi_j) \bigg\rangle_{j+1,\ldots,d}  \right) \\
= &\bigg\langle N(u_0), \sum_{\alpha_0,\ldots,\alpha_{j-2}=1 }^{r_0,\ldots,r_{j-2}} \psi_1(\alpha_0,\alpha_1) \cdots \psi_{j-1}(\alpha_{j-1},\xi_{j-1}) \varphi_j(\xi_j) \bigg\rangle_{1,\ldots,j-1,j+1,\ldots,d} \\
& - \frac{1}{2} \sum_{\alpha_j} \lambda_{\xi_j \alpha_j}^{(j)} \psi_j(\xi_{j-1},\alpha_j).
\end{aligned}
\end{equation}
Taking $\langle \cdot, \psi_j(\alpha_{j-1},\xi'_j) \rangle_{L^2_{\tau \times \mu_j}(\mathbb{N} \times \Omega_j)}$ of the previous equation and utilizing the constraints we find
\begin{equation}
\label{multiplier}
\begin{aligned}
\lambda_{\xi_j \xi'_j}^{(j)} = 2\biggr[ &\sum_{\alpha_0,\ldots,\alpha_{j-1} =1 }^{r_0,\ldots,r_{j-1}} \bigg\langle N(u), 
\psi_1(\alpha_0,\alpha_1) \cdots \psi_{j-1}(\alpha_{j-2},\xi_{j-1}) \psi_j(\alpha_{j-1}, \xi'_j) \varphi_j(\xi_j) \bigg\rangle_{1,\ldots,d} \\
&- \bigg\langle \frac{\partial {\varphi}_j(\xi'_j)}{\partial s}, \varphi_j(\xi_j) \bigg\rangle_{j+1,\ldots,d} \biggr]
\end{aligned}
\end{equation}
Plugging \eqref{multiplier} into \eqref{DO_derivation_1} 
and simplifying we obtain
\begin{equation}
\begin{aligned}
&\sum_{\alpha_j=1}^{r_j} \frac{\partial {\psi}_j(\xi_{j-1},\alpha_j) }{\partial s}\bigg\langle \varphi_j(\alpha_j),\varphi_j(\xi_j) \bigg\rangle_{j+1,\ldots,d} \\
=&  \sum_{\alpha_0,\ldots,\alpha_{j-2}=1}^{r_0,\ldots,r_{j-2}} \bigg\langle 
N(u), \psi_1(\alpha_0,\alpha_1) \cdots \psi_{j-1}(\alpha_{j-2},\xi_{j-1}) \varphi_j(\xi_j) \bigg\rangle_{1,\ldots,j-1,j+1,\ldots,d} \\
&- \sum_{\alpha_0,\ldots,\alpha_j=1}^{r_0,\ldots,r_j} \psi_j(\xi_{j-1},\alpha_j) \bigg\langle N(u),\psi_1(\alpha_0,\alpha_1) \cdots \psi_{j-1}(\alpha_{j-2},\xi_{j-1}) \psi_j(\alpha_{j-1},\alpha_j) \varphi_j(\xi_j) \bigg\rangle_{1,\ldots,d}.
\end{aligned}
\end{equation}
Using the matrix vector notation for tensor cores and 
inverting the auto-correlation matrix on the left hand side yields 
the equation for $\partial \Psi_j / \partial s$ in \eqref{DO-FTT_system}.
For $j = d$ we obtain 
\begin{equation}
\begin{aligned}
&\left[\delta_{\frac{\partial {\psi}_d(\xi_{0},\xi_1)}{\partial s}} \mathcal{A}\right] \eta_d (\xi_{d-1} , \xi_d) 
\\
&= 2\biggl\langle  \frac{\partial }{\partial s} \left[ \sum_{\alpha_0,\ldots,\alpha_d=1}^{r} \psi_1(\alpha_0,\alpha_1) \cdots \psi_d(\alpha_{d-1},\alpha_d) \right] - N(u) , \\
&\sum_{\alpha_0,\ldots,\alpha_{d-2}=1}^{r_0,\ldots,r_{d-1}} \psi_1(\alpha_0,\alpha_1)\cdots \psi_{d-1}(\alpha_{d-2},\xi_{d-1}) \eta_d(\xi_{d-1},\xi_d) \biggr\rangle_{1,2,\ldots,d} \\
&= 0, \qquad \forall \eta_d(\xi_{d-1},\xi_d) \in L^2_{\mu_d}(\Omega_1),
\end{aligned}
\end{equation}
whence the fundamental lemma of calculus of variations implies 
\begin{equation}
\begin{aligned}
&\biggl\langle  \frac{\partial }{\partial s} \left[ \sum_{\alpha_0,\ldots,\alpha_d=1}^{r} \psi_1(\alpha_0,\alpha_1) \cdots \psi_d(\alpha_{d-1},\alpha_d) \right] - N(u) , \\
&\sum_{\alpha_0,\ldots,\alpha_{d-2}=1}^{r_0,\ldots,r_{d-1}} \psi_1(\alpha_0,\alpha_1)\cdots \psi_{d-1}(\alpha_{d-2},\xi_{d-1}) \biggr\rangle_{1,\ldots,d-1}= 0.
\end{aligned}
\end{equation}
Rearranging terms we obtain 
\begin{equation}
\begin{aligned}
&\left\langle \frac{\partial }{\partial s} \left[ \sum_{\alpha_0,\ldots,\alpha_d=1}^{r} \psi_1(\alpha_0,\alpha_1) \cdots \psi_d(\alpha_{d-1},\alpha_d) \right], 
\sum_{\alpha_0,\ldots,\alpha_{d-2}=1}^{r_0,\ldots,r_{d-1}} \psi_1(\alpha_0,\alpha_1)\cdots \psi_{d-1}(\alpha_{d-2},\xi_{d-1}) \right\rangle_{1,\ldots,d-1} \\
&=\left\langle N(u), \sum_{\alpha_0,\ldots,\alpha_{d-2}=1}^{r_0,\ldots,r_{d-1}} \psi_1(\alpha_0,\alpha_1)\cdots \psi_{d-1}(\alpha_{d-2},\xi_{d-1}) \right\rangle_{1,\ldots,d-1}.
\end{aligned}
\end{equation}
Using the dynamic and static orthogonality constraints 
we obtain 
\begin{equation}
\begin{aligned}
\frac{\partial \psi_d(\xi_{d-1},1)}{\partial s} = \left\langle N(u), \sum_{\alpha_0,\ldots,\alpha_{d-2}=1}^{r_0,\ldots,r_{d-1}} \psi_1(\alpha_0,\alpha_1)\cdots \psi_{d-1}(\alpha_{d-2},\xi_{d-1}) \right\rangle_{1,\ldots,d-1}.
\end{aligned}
\end{equation}
Writing this expression in matrix-vector notation the desired equation for 
$\partial \Psi_d / \partial s$ is obtained.

%\section*{References}
\bibliographystyle{plain}
\bibliography{bibliography_file}

\end{document}